\documentclass[11pt, twoside]{article}
\usepackage{amsfonts}
\usepackage{ulem}
\usepackage{amssymb}
\usepackage{amsmath}
\usepackage{amsthm}
\usepackage{xcolor}
\usepackage{mathrsfs}
\usepackage[misc]{ifsym}
\usepackage{verbatim}

\allowdisplaybreaks

\allowdisplaybreaks

\newtheorem{theorem}{Theorem}[section]
\newtheorem{lemma}[theorem]{Lemma}
\newtheorem{corollary}[theorem]{Corollary}
\newtheorem{proposition}[theorem]{Proposition}
\theoremstyle{definition}

\newtheorem{remark}[theorem]{Remark}

\newtheorem{definition}[theorem]{Definition}
\numberwithin{equation}{section}

\numberwithin{equation}{section}

\numberwithin{equation}{section}

\begin{document}

\arraycolsep=1pt

\title{\Large\bf Endpoint estimates of discrete fractional operators on discrete weighted Lebesgue spaces \footnotetext{\hspace{-0.35cm} {\it 2020
Mathematics Subject Classification}. {46B45, 42B20, 42B25.}
\endgraf{\it Key words and phrases.} weight, endpoint estimate, discrete fractional maximal operator, discrete Riesz potential.
\endgraf This project is supported by the National Natural Science Foundation of China (Grant Nos. 12261083 \& 12161083) and Xinjiang Key Laboratory of Applied Mathematics (No. XJDX1401).
\endgraf $^\ast$\,Corresponding author.
\endgraf \quad~{Xiong Hu}
\endgraf \quad~{xionghu92@qq.com}
\medskip
\endgraf \quad~{Xuebing Hao}
\endgraf \quad~{1659230998@qq.com}
\medskip
\endgraf ${\textrm{\Letter}}$ Baode Li
\endgraf \quad~{baodeli@xju.edu.cn}
\medskip
\endgraf College of Mathematics and System Sciences, Xinjiang University, Urumqi, 830017, P. R. China
}}
\author{Xiong Hu, Xuebing Hao and Baode Li$^\ast$}
\date{ }
\maketitle

\vspace{-0.8cm}

\begin{center}
\begin{minipage}{13cm}\small
{\noindent{\bf Abstract.}
Let $0<\alpha<1$ and $\frac{1}{q}=1-\alpha$. We first obtain that the function $\omega :\mathbb{Z} \rightarrow (0,\infty)$ belongs to weight class of $\mathcal{A} (1,q)(\mathbb{Z})$ if and only if discrete fractional maximal operator $M_{\alpha}$ or discrete Riesz potential $I_\alpha$ is bounded from $l_{\omega}^{1}(\mathbb{Z})$ to $l_{\omega^q}^{q,weak}(\mathbb{Z})$. Then for $p=\frac{1}{\alpha}$, we further obtain that the function $\omega$ belongs to weight class of $\mathcal{A} (p,\infty)(\mathbb{Z})$ if and only if discrete Riesz potential $I_\alpha$ has a property resembling discrete bounded mean oscillation. Moreover, we give another simple proof of $I_{\alpha}:l_{\omega ^p}^{p}(\mathbb{Z}) \rightarrow l_{\omega ^q}^{q}(\mathbb{Z})$ for $\omega \in \mathcal{A}(p,q)(\mathbb{Z})$, $1<p<\frac{1}{\alpha}$ and $\frac{1}{q}=\frac{1}{p}-\alpha$. As applications, more weighted norm inequalities for $M_{\alpha}$ and $I_\alpha$ are established when $\omega \in \mathcal{A}(1,q)(\mathbb{Z})$ or $\omega \in \mathcal{A}(p,\infty)(\mathbb{Z})$, and some of them are new even in continuous setting.}
\end{minipage}
\end{center}

\section{Introduction\label{s1}}
Let $0<\alpha<1$. Define Riesz potential $\mathcal{I}_{\alpha}f(x) :=\int_{\mathbb{R}^n}{f(y)|x-y|^{\alpha -n}dy}$ and fractional maximal function $\mathcal{M}_{\alpha}f(x) :=\underset{Q}{\mathrm{sup}}\ |Q|^{-1+{{\alpha}/{n}}}\int_Q{|f(y)|dy}$, where the supremum is taken over all cubes $Q$ with center at $x$ and $|Q|$ denotes the measure of $Q$. Suppose that $1<p,q<\infty$ and $\omega(x)$ is a nonnegative locally integrable function on $\mathbb{R}^n$. Define $\omega \in \mathcal{A}(p,q)$, if
\begin{equation*}
\underset{Q}{\mathrm{sup}}\left( \frac{1}{|Q|}\int_Q{\omega}(x)^qdx \right) ^{\frac{1}{q}}\left( \frac{1}{|Q|}\int_Q{\omega}(x)^{-p\prime}dx \right) ^{\frac{1}{p\prime}}<\infty,
\end{equation*}
where the supremum is taken over all cubes $Q$ in $\mathbb{R}^n$ and $p^{\prime}:={p}/(p-1)$. In 1974, Muckenhoupt and Wheeden \cite{1974m} characterized the weighted strong type inequality for fractional maximal operator and Riesz potential in terms of weight class of $\mathcal{A}(p,q)$, where $1<p<\frac{1}{\alpha}$ and $\frac{1}{q}=\frac{1}{p}- \alpha$. They showed that the inequality
\begin{equation*}
\left\| T_{\alpha}f \right\| _{L_{\omega ^q}^{q}}\le C\left\| f \right\| _{L_{\omega ^p}^{p}},
\end{equation*}
where $T_{\alpha}$ is the operator $\mathcal{I}_{\alpha}$ or $\mathcal{M}_{\alpha}$, holds if and only if $\omega \in \mathcal{A}(p,q)$. Moreover, in 1975, Welland \cite{welland} gave a simple proof of weighted strong type inequality for Riesz potential in terms of $\mathcal{A} ( p,q )$, and also gave certain results for $p=1$ and $q=\infty$.

It is well known that fractional maximal operator and Riesz potential play important roles in various branches of analysis, including potential theory, harmonic analysis, Sobolev spaces, partial differential equations and can be treated as special singular integrals. That is why it is important to study their boundedness between different spaces; see \cite{adams, burtseva, norm, 1928, 1974m, lsHardy22, stein1, stein2}.

Modern discrete analogues in harmonic analysis developed out of work of Jean Bourgain in the late 1980s and early 1990s, and Eli Stein and his longterm collaborators continuing into this century. For example, the study of boundedness of discrete operators on $l^p$ analogues for $L^p$-boundedness has been considered by some authors; see \cite{anupindi, wiener20, guliyev, krause, 2009p}.

The aim of this article is to obtain some weighted estimates for discrete fractional maximal operator and discrete Riesz potential. Let begin with some related definitions.

Let $S_{m,N}:=\{k\in \mathbb{Z} :|k-m|\leq N \}$, $m\in \mathbb{Z}$, $N\in \mathbb{N}$. Then $|S_{m,N}|:=2N+1$$-$the cardinality of $S_{m,N}$.

\begin{definition}\label{d1.1}\cite[Definition 2.1]{anupindi}
Let $0\leq \alpha <1$ and $x=\{x(k)\} _{k\in \mathbb{Z}}\subset \mathbb{R}$ be a sequence. The {\it discrete fractional maximal operator} $M_{\alpha}$ is defined by
\begin{equation*}
M_{\alpha}x\left( m \right) :=\underset{N\in \mathbb{N}}{\mathrm{sup}}\ \frac{1}{\left| S_{m,N} \right|^{1-\alpha}}\sum_{k\in S_{m,N}}{\left| x\left( k \right) \right|},\enspace m\in \mathbb{Z}.
\end{equation*}
\end{definition}

\begin{definition}\label{d1.2}\cite[Page 8]{Schwanke}
Let $0<\alpha <1$ and $x=\{ x( k ) \} _{k\in \mathbb{Z}}\subset \mathbb{R}$ be a sequence. The {\it discrete Riesz potential} $I_{\alpha}$ is defined by
\begin{equation*}
I_{\alpha}x\left( k \right) :=\sum_{i\in \mathbb{Z} \setminus \left\{ k \right\}}{\frac{x\left( i \right)}{\left| k-i \right|^{1-\alpha}}},\enspace k\in \mathbb{Z}.
\end{equation*}
\end{definition}

In 2024, Hao-Yang-Li \cite{hao2} obtained the weighted strong type inequalities for discrete fractional maximal operator $M_{\alpha}$ and discrete Riesz potential $I_{\alpha}$ on discrete weighted Lebesgue spaces in terms of $\mathcal{A}(p,q)(\mathbb{Z})$ (see Definition 2.2); see the following Theorems \ref{t1.3} and \ref{t1.4}.

\begin{theorem}\label{t1.3}\cite[Theorem 3.3]{hao2}
Let $0<\alpha<1$, $1<p<\frac{1}{\alpha}$ and $\frac{1}{q}=\frac{1}{p}- \alpha$. Then $M_{\alpha}$ is bounded from $l_{\omega^p}^{p}$ to $l_{\omega ^q}^{q}$ if and only if $\omega \in \mathcal{A}(p,q)$, and there exists a positive constant $C$ such that
\begin{equation}\label{eq1.1}
\left\| M_{\alpha}x \right\| _{l_{\omega ^q}^{q}}\leq C\left\| x \right\| _{l_{\omega ^p}^{p}}.
\end{equation}
\end{theorem}

\begin{theorem}\label{t1.4}\cite[Theorem 4.4]{hao2}
Let $0<\alpha<1$, $1<p<\frac{1}{\alpha}$ and $\frac{1}{q}=\frac{1}{p}- \alpha$.
\begin{itemize}
\item[\rm(i)] If $\omega \in \mathcal{A}(p,q)$, then $I_\alpha$ is bounded from $l_{\omega^p}^{p}$ to $l_{\omega^q}^{q}$.
\item[\rm(ii)] Conversely, if $\omega$ is a discrete weighted sequence which is monotonic on $\mathrm{supp}\,\omega$ and $I_\alpha$ is bounded from $l_{\omega ^p}^{p}$ to $l_{\omega ^q}^{q}$, then $\omega \in \mathcal{A}(p,q)$.
\end{itemize}
\end{theorem}

The weighted endpoint estimates for $M_{\alpha}$ and $I_{\alpha}$ have not been studied. Inspired by \cite{1974m}, \cite{welland} and \cite{hao2}, the following natural question appeared:

Can we prove the endpoint estimates of discrete fractional maximal operator $M_{\alpha}$ and discrete Riesz potential $I_\alpha$ on discrete weighted Lebesgue spaces for the case $\omega\in \mathcal{A}(1,q)(\mathbb{Z})$ or $\omega\in \mathcal{A}(p,\infty)(\mathbb{Z})$?

Our aim is to give an affirmative answer to the above question which promotes the development of the applications of discrete Muckenhoupt weights, discrete Lebesgue spaces to discrete harmonic analysis.

This article is organized as follows. In Sect.\,\ref{s2}, we mainly recall discrete Muckenhoupt weight classes and give some related properties. In Sect.\,\ref{s3}, inspired by the boundedness of fractional maximal operator and Riesz potential on weighted Lebesgue spaces of Muckenhoupt and Wheeden \cite{1974m}, we obtain that $\omega \in \mathcal{A}(1,q)(\mathbb{Z})$ if and only if discrete fractional maximal operator $M_{\alpha}$ or discrete Riesz potential $I_\alpha$ is bounded from $l_{\omega}^{1}(\mathbb{Z})$ to $l_{\omega^q}^{q,weak}(\mathbb{Z})$, $\frac{1}{q}=1-\alpha$. Here it is worth noticing that we need the assumption ``$\omega$ is monotonic on $\mathrm{supp}\,\omega$'' to guarantee $I_\alpha: l_{\omega^p}^{p}(\mathbb{Z}) \rightarrow l_{\omega^q}^{q}(\mathbb{Z}) \Rightarrow \omega\in\mathcal{A}(p,q)(\mathbb{Z})$ (see Theorem \ref{t1.4}(ii)). But we don't need that assumption to guarantee $I_\alpha: l_{\omega}^{1} \rightarrow l_{\omega^q}^{q,weak} \Rightarrow \omega\in\mathcal{A}(1,q)(\mathbb{Z})$ (see Theorem \ref{t3.5} below). We further obtain that $\omega \in \mathcal{A}(p,\infty)(\mathbb{Z})$ if and only if discrete Riesz potential $I_\alpha$ has a property resembling discrete bounded mean oscillation. As applications, in Sect.\,\ref{s4}, more weighted norm inequalities for $M_{\alpha}$ and $I_\alpha$ are established with $\omega \in \mathcal{A}(1,q)(\mathbb{Z})$ or $\omega \in \mathcal{A}(p,\infty)(\mathbb{Z})$. For example, if $\omega \in \mathcal{A}(1,q_0)(\mathbb{Z})$, $\frac{1}{q_0}=1-\alpha$, then $M_{\alpha}:l_{\omega ^p}^{p}(\mathbb{Z}) \rightarrow l_{\omega ^q}^{q}(\mathbb{Z})$, where $(p,q)$ is near $(1,q_0)$ (see Theorem \ref{t4.1}). This seems new even in continuous setting. Moreover, comparing to the proof of \cite[Theorem 4.4(i)]{hao2}, i.e., Theorem 1.4(i), we give another simple proof of $I_{\alpha}:l_{\omega ^p}^{p}(\mathbb{Z}) \rightarrow l_{\omega ^q}^{q}(\mathbb{Z})$, and we correct the proof of \cite[Theorem 4.4(ii)]{hao2}, i.e., Theorem \ref{t1.4}(ii), where $\omega \in \mathcal{A}(p,q)(\mathbb{Z})$, $1<p<\frac{1}{\alpha}$ and $\frac{1}{q}=\frac{1}{p}-\alpha$.

\section*{Notations}

\ \,\ \ $-\mathbb{Z}:$ the set of integers;

$-\mathbb{Z} _+:=\{ 1,2,3,\dots \}$;

$-\mathbb{N}:=\{ 0,1,2,\dots \}$;

$-J:$ the interval of finite consecutive integers, i.e., $J=\{ m,m+1,\dots ,m+n \}$, $m\in \mathbb{Z} $, $n\in \mathbb{Z_+} $;

$-\mathcal{X} _J:$ the characteristic function of interval $J$;

$-\lfloor k \rfloor:$ the nearest integer less than or equal to $k$;

$-\lambda S_{m,N}:=\{ k\in \mathbb{Z} :|k-m|\leq \lambda N \}$, where $\lambda\in \mathbb{Z_+}$, $m\in \mathbb{Z}$ and $N\in \mathbb{N}$;

$-\omega (J) :=\sum_{k\in J}{\omega (k)}$, for any interval $J\subset\mathbb{Z}$;

$-\mathrm{supp}\,x:=\{ k\in \mathbb{Z} :x(k) \ne 0 \}$;

$-E^c:$ the complementary set of $E$;

$-p^{\prime}:$ the conjugate exponent of $p$, namely, $\frac{1}{p}+\frac{1}{p^{\prime}}=1$;

$-C:$ a positive constant which is independent of the main parameters, but it may vary from line to line.




\section{Preliminaries \label{s2}}
In this section, we mainly recall discrete Muckenhoupt weight classes and give some related properties.

\begin{definition}\label{d2.1}\cite[Definition 2.7]{hao1}
A discrete weight $\omega:\mathbb{Z}\rightarrow(0,\infty)$ is said to belong to the {\it discrete Muckenhoupt class} $\mathcal{A}_1:=\mathcal{A}_1(\mathbb{Z})$ if
\begin{equation*}
\left\| \omega \right\| _{\mathcal{A} _1\left( \mathbb{Z} \right)}:=\underset{m\in \mathbb{Z} ,N\in \mathbb{N}}{\mathrm{sup}}\,\,\frac{1}{\left| S_{m,N} \right|}\Bigg( \frac{1}{\underset{k\in S_{m,N}}{\mathrm{inf}}\,\omega \left( k \right)}\sum_{k\in S_{m,N}}{\omega \left( k \right)} \Bigg) <\infty.
\end{equation*}
For $1<p<\infty$, a discrete weight $\omega:\mathbb{Z}\rightarrow(0,\infty)$ is said to belong to the {\it discrete Muckenhoupt class} $\mathcal{A}_p:=\mathcal{A}_p( \mathbb{Z})$ if
\begin{equation*}
\left\| \omega \right\| _{\mathcal{A} _p\left( \mathbb{Z} \right)}:=\underset{m\in \mathbb{Z} ,N\in \mathbb{N}}{\mathrm{sup}}\bigg( \frac{1}{\left| S_{m,N} \right|}\sum_{k\in S_{m,N}}{\omega \left( k \right)} \bigg) \bigg( \frac{1}{\left| S_{m,N} \right|}\sum_{k\in S_{m,N}}{\omega \left( k \right) ^{\frac{-1}{p-1}}} \bigg) ^{p-1}<\infty.
\end{equation*}
Define $\mathcal{A}_{\infty}:=\bigcup_{1\leq p<\infty}{\mathcal{A}_p}$.
\end{definition}

The following definition is a discrete variant of Muckenhoupt $\mathcal{A}(p,q)$ class (see \cite[Page 261]{1974m}), which is also introduced in \cite[Definition 2.6]{hao2}.

\begin{definition}\label{d2.2}
A discrete weight $\omega:\mathbb{Z}\rightarrow(0,\infty)$ is said to belong to $\mathcal{A}(p,q)$ on $\mathbb{Z}$ for $1<p,q<\infty$ if
\begin{equation*}
\left\| \omega \right\| _{\mathcal{A} \left( p,q \right) \left( \mathbb{Z} \right)}:=\underset{m\in \mathbb{Z} ,N\in \mathbb{N}}{\mathrm{sup}}\bigg( \frac{1}{\left| S_{m,N} \right|}\sum_{k\in S_{m,N}}{\omega \left( k \right) ^q} \bigg) ^{\small{\frac{1}{q}}}\bigg( \frac{1}{\left| S_{m,N} \right|}\sum_{k\in S_{m,N}}{\omega \left( k \right) ^{-p^{\prime}}} \bigg) ^{\small{\frac{1}{p^{\prime}}}}<\infty.
\end{equation*}
Particularly, when $p=1$,
\begin{equation*}
\left\| \omega \right\| _{\mathcal{A} \left( 1,q \right) \left( \mathbb{Z} \right)}:=\underset{m\in \mathbb{Z} ,N\in \mathbb{N}}{\mathrm{sup}}\bigg( \frac{1}{\left| S_{m,N} \right|}\sum_{k\in S_{m,N}}{\omega \left( k \right) ^q} \bigg) ^{\small{\frac{1}{q}}}\bigg( \underset{k\in S_{m,N}}{\mathrm{sup}}\,\,\frac{1}{\omega \left( k \right)} \bigg) <\infty.
\end{equation*}
When $q=\infty$,
\begin{equation*}
\left\| \omega \right\| _{\mathcal{A} \left( p,\infty \right) \left( \mathbb{Z} \right)}:=\underset{m\in \mathbb{Z} ,N\in \mathbb{N}}{\mathrm{sup}}\bigg( \underset{k\in S_{m,N}}{\mathrm{sup}}\,\,\omega \left( k \right) \bigg) \bigg( \frac{1}{\left| S_{m,N} \right|}\sum_{k\in S_{m,N}}{\omega \left( k \right) ^{-p^{\prime}}} \bigg) ^{\small{\frac{1}{p^{\prime}}}}<\infty.
\end{equation*}
\end{definition}

\begin{proposition}\label{p2.3}
Let $0<\alpha<1$.
\begin{itemize}
\item[$\rm(i)$]\cite[Page 139]{lu}
If $1<p,q<\infty$, then $\omega \in \mathcal{A}(p,q)\Longleftrightarrow \omega ^q\in \mathcal{A} _{1+\frac{q}{p\prime}}\Longleftrightarrow \omega ^{-p\prime}\in \mathcal{A}_{1+\frac{p\prime}{q}}$. Particularly, $\omega \in \mathcal{A}(1,q)\Longleftrightarrow \omega ^q\in \mathcal{A} _1$, $\omega \in \mathcal{A}(p,\infty)\Longleftrightarrow \omega ^{-p\prime}\in \mathcal{A} _1$.
\item[$\rm(ii)$]\cite[Proposition 2.16(iii)]{hao1}\ (Discrete reverse H\"older's inequality)
If $\omega \in \mathcal{A} _p$ for some $1\leq p<\infty$, then there exists a constant $r \in (1, \infty)$ such that $\omega \in RH_r$, i.e., for any symmetric interval $S_{m,N}\subset\mathbb{Z}$, it holds true that
\begin{equation*}
\bigg( \frac{1}{\left| S_{m,N} \right|}\sum_{k\in S_{m,N}}{\omega \left( k \right) ^r} \bigg) ^{\frac{1}{r}}\leq \frac{C}{\left| S_{m,N} \right|}\sum_{k\in S_{m,N}}{\omega \left( k \right)}.
\end{equation*}
\item[$\rm(iii)$]
Let $1<p<\frac{1}{\alpha}$ and $\frac{1}{q}=\frac{1}{p}- \alpha$. If $\omega \in \mathcal{A}(p,q)$, then there exist real numbers $p_1$, $p_2$, $q_1$ and $q_2$ such that $1<p_1<p<p_2<\frac{1}{\alpha}$, $\frac{1}{q_i}=\frac{1}{p_i}-\alpha$ and $\omega \in \mathcal{A}(p_i,q_i)$, $i=1,2$.
\item[$\rm(iv)$]
Let $\frac{1}{q_0}=1-\alpha$. If $\omega\in \mathcal{A}(1,q_0)$, then there exist real numbers $p_1\in(1,\frac{1}{\alpha})$ and $q_1\in(q_0,\infty)$ such that $\frac{1}{q_1}=\frac{1}{p_1}-\alpha$ and $\omega \in \mathcal{A}(p,q)$, where $\frac{1}{p}=t+\frac{1-t}{p_1}$, $\frac{1}{q}=\frac{t}{q_0}+\frac{1-t}{q_1}$, $t\in[0,1]$.
\item[$\rm(v)$]
Let $p_0=\frac{1}{\alpha}$. If $\omega\in \mathcal{A}(p_0,\infty)$, then there exist real numbers $p_2\in(1,p_0)$ and $q_2\in(1,\infty)$ such that $\frac{1}{q_2}=\frac{1}{p_2}-\alpha$ and $\omega \in \mathcal{A}(p,q)$, where $\frac{1}{p}=\frac{t}{p_0}+\frac{1-t}{p_2}$, $\frac{1}{q}=\frac{1-t}{q_2}$, $t\in[0,1]$.
\end{itemize}
\end{proposition}

\begin{proof}
{\rm(i)}
By the definitions of $\mathcal{A}_p$ and $\mathcal{A}(p,q)$, and discrete H\"older's inequality, we immediately obtain (i). The details are omitted and its corresponding results on real line can be found in \cite[Page 139]{lu}.

{\rm(iii)}
The continuous version of (iii) is also mentioned in the proof of \cite[Lemma 1]{welland} but without complete proof. For the sake of completeness, the proof is given as follows. Notice that $\omega \in \mathcal{A}(p,q)$, then by (i) and (ii), we obtain that $\omega ^{-p\prime}\in \mathcal{A}_{1+\frac{p\prime}{q}}$ and there exists $r_1>1$ such that $\omega ^{-p\prime}\in RH_{r_1}$. Pick $p_1\in(1,p)$ and $q_1\in(1,q)$ such that $\frac{{p_1}^{\prime}}{p^{\prime}}=r_1$ and $\frac{1}{q_1}=\frac{1}{p_1}-\alpha$. Then from this, discrete H\"older's inequality and (ii), it follows that
\begin{align*}
&\left( \frac{1}{\left| S_{m,N} \right|}\sum_{k\in S_{m,N}}{\omega \left( k \right) ^{q_1}} \right) ^{\small{\frac{1}{q_1}}}\left( \frac{1}{\left| S_{m,N} \right|}\sum_{k\in S_{m,N}}{\omega \left( k \right) ^{-{p_1}^{\prime}}} \right) ^{\small{\frac{1}{{p_1}^{\prime}}}}
\\
\leq& \left( \frac{1}{\left| S_{m,N} \right|}\sum_{k\in S_{m,N}}{\omega \left( k \right) ^q} \right) ^{\small{\frac{1}{q}}}\left( \frac{1}{\left| S_{m,N} \right|}\sum_{k\in S_{m,N}}{\omega \left( k \right) ^{-p^{\prime}r_1}} \right) ^{\frac{1}{p^{\prime}r_1}}
\\
\leq& C\left( \frac{1}{\left| S_{m,N} \right|}\sum_{k\in S_{m,N}}{\omega \left( k \right) ^q} \right) ^{\small{\frac{1}{q}}}\left( \frac{1}{\left| S_{m,N} \right|}\sum_{k\in S_{m,N}}{\omega \left( k \right) ^{-p^{\prime}}} \right) ^{\small{\frac{1}{p^{\prime}}}}\leq C\left\| \omega \right\| _{\mathcal{A} \left( p,q \right) \left( \mathbb{Z} \right)},
\end{align*}
which implies that $\omega \in \mathcal{A}(p_1,q_1)$.

Similarly, since $\omega \in \mathcal{A}(p,q)$, then by (i) and (ii), we obtain that $\omega ^q\in \mathcal{A} _{1+\frac{q}{p\prime}}$ and there exists $r_2>1$ such that $\omega ^q\in RH_{r_2}$. Pick $p_2\in(p,\frac{1}{\alpha})$ and $q_2\in(q,\infty)$ such that $\frac{q_2}{q}=r_2$ and $\frac{1}{q_2}=\frac{1}{p_2}-\alpha$. Then from this, discrete H\"older's inequality and (ii), it follows that
\begin{align*}
&\left( \frac{1}{\left| S_{m,N} \right|}\sum_{k\in S_{m,N}}{\omega \left( k \right) ^{q_2}} \right) ^{\small{\frac{1}{q_2}}}\left( \frac{1}{\left| S_{m,N} \right|}\sum_{k\in S_{m,N}}{\omega \left( k \right) ^{-{p_2}^{\prime}}} \right) ^{\small{\frac{1}{{p_2}^{\prime}}}}
\\
\leq& \left( \frac{1}{\left| S_{m,N} \right|}\sum_{k\in S_{m,N}}{\omega \left( k \right) ^{qr_2}} \right) ^{\frac{1}{qr_2}}\left( \frac{1}{\left| S_{m,N} \right|}\sum_{k\in S_{m,N}}{\omega \left( k \right) ^{-p^{\prime}}} \right) ^{\small{\frac{1}{p^{\prime}}}}
\\
\leq& C\left( \frac{1}{\left| S_{m,N} \right|}\sum_{k\in S_{m,N}}{\omega \left( k \right) ^q} \right) ^{\small{\frac{1}{q}}}\left( \frac{1}{\left| S_{m,N} \right|}\sum_{k\in S_{m,N}}{\omega \left( k \right) ^{-p^{\prime}}} \right) ^{\small{\frac{1}{p^{\prime}}}}\leq C\left\| \omega \right\| _{\mathcal{A} \left( p,q \right) \left( \mathbb{Z} \right)},
\end{align*}
which implies that $\omega \in \mathcal{A}(p_2,q_2)$.

{\rm(iv)}
If $\omega \in \mathcal{A}(1,q_0)$, then by (i) and (ii), we obtain that $\omega ^{q_0}\in \mathcal{A} _1$ and there exists $r_0>1$ such that $\omega ^{q_0}\in RH_{r_0}$. Pick $p_1\in(1,\frac{1}{\alpha})$ and $q_1\in(q_0,\infty)$ such that $\frac{q_1}{q_0}=r_0$ and $\frac{1}{q_1}=\frac{1}{p_1}-\alpha$. Then from this and (ii), it follows that
\begin{align}\label{eq2.1}
&\left( \frac{1}{\left| S_{m,N} \right|}\sum_{k\in S_{m,N}}{\omega \left( k \right) ^{q_1}} \right) ^{\small{\frac{1}{q_1}}}\left( \frac{1}{\left| S_{m,N} \right|}\sum_{k\in S_{m,N}}{\omega \left( k \right) ^{-{p_1}^{\prime}}} \right) ^{\small{\frac{1}{{p_1}^{\prime}}}}\nonumber
\\
\leq& \left( \frac{1}{\left| S_{m,N} \right|}\sum_{k\in S_{m,N}}{\omega \left( k \right) ^{q_0r_0}} \right) ^{\small{\frac{1}{q_0r_0}}}\left( \underset{k\in S_{m,N}}{\mathrm{sup}}\,\,\frac{1}{\omega \left( k \right)} \right)\nonumber
\\
\leq& C\left( \frac{1}{\left| S_{m,N} \right|}\sum_{k\in S_{m,N}}{\omega \left( k \right) ^{q_0}} \right) ^{\small{\frac{1}{q_0}}}\left( \underset{k\in S_{m,N}}{\mathrm{sup}}\,\,\frac{1}{\omega \left( k \right)} \right) \leq C\left\| \omega \right\| _{\mathcal{A} \left( 1,q_0 \right) \left( \mathbb{Z} \right)},
\end{align}
which implies that $\omega \in \mathcal{A}(p_1,q_1)$. Moreover, for any $t\in [0,1]$, pick $p\in [1,p_1]$ and $q\in [q_0,q_1]$ such that $\frac{1}{p}=t+\frac{1-t}{p_1}$ and $\frac{1}{q}=\frac{t}{q_0}+\frac{1-t}{q_1}$. Then for $r:=\frac{q}{q_0}\in (1,r_0]$, by discrete H\"older's inequality and $\omega ^{q_0}\in RH_{r_0}$, we have $\omega ^{q_0}\in RH_r$. From this and repeating the estimate of (\ref{eq2.1}), it follows that $\omega \in \mathcal{A}(p,q)$.

{\rm(v)}
The proof is almost the same as that of (iv), so the details are omitted.
\end{proof}

\begin{definition}\label{d2.4}\cite[Page 3]{fu}
Let $1\leq p<\infty$ and $\omega:\mathbb{Z}\rightarrow(0,\infty)$ be a discrete weight. The discrete weighted weak Lebesgue space $l^{p,weak}_\omega:=l^{p,weak}_\omega(\mathbb{Z})$ is the set of sequences $x=\{x(k)\}_{k\in\mathbb{Z}}\subset\mathbb{R}$ such that
\begin{equation*}
\|x\|_{l^{p,weak}_\omega}:=\sup_{\lambda>0}\,\lambda\left[\omega(\{k\in\mathbb{Z}:|x(k)|>\lambda\})\right]^\frac{1}{p}<\infty.
\end{equation*}
\end{definition}

\section{The endpoint estimates of discrete Riesz potential on discrete weighted Lebesgue spaces \label{s3}}
In this section, we discuss the sufficient and necessary results of discrete fractional maximal operator and discrete Riesz potential on discrete weighted Lebesgue spaces at the cases $\omega\in \mathcal{A}(1,q)$ and $\omega\in \mathcal{A}(p,\infty)$.

\begin{definition}\label{d3.1}\cite[Definition 2.1]{anupindi}
Let $0\leq \alpha <1$ and $x=\{x(k)\} _{k\in \mathbb{Z}}\subset \mathbb{R}$ be a sequence. The {\it noncentral discrete fractional maximal operator} $\overline{M}_{\alpha}$ is defined by
\begin{equation*}
\overline{M}_{\alpha}x\left( m \right) :=\underset{J\ni m}{\mathrm{sup}}\ \frac{1}{\left| J \right|^{1-\alpha}}\sum_{k\in J}{\left| x\left( k \right) \right|},\enspace m\in \mathbb{Z},
\end{equation*}
where $J$ is any interval of finite consecutive integers.
\end{definition}

The following Theorem \ref{t3.2} is a sufficient and necessary result for weight class of $\mathcal{A}(1,q)$.

\begin{theorem}\label{t3.2}
Let $0<\alpha <1$ and $\frac{1}{q}=1-\alpha$. Then $\omega\in \mathcal{A}(1,q)$ if and only if $M_{\alpha}$ is bounded from $l_{\omega}^{1}$ to $l_{\omega^q}^{q,weak}$, i.e., there exists a positive constant $C$ such that, for every $\lambda >0$,
\begin{equation}\label{eq3.1}
\left( \sum_{\{k\in \mathbb{Z} :M_{\alpha}x(k)>\lambda \}}{\omega}(k)^q \right) ^{\frac{1}{q}}\leq \frac{C}{\lambda}\sum_{k\in \mathbb{Z}}{\left| x(k) \right|}\omega (k).
\end{equation}
\end{theorem}

To prove Theorem \ref{t3.2}, we need the following two lemmas.

\begin{lemma}\label{l3.3}\cite[Remark 3.2]{hao2}
Let $x=\{x(m)\}_{m\in \mathbb{Z}}\subset \mathbb{R}$ be a sequence and $0\leq \alpha <1$. Then $M_{\alpha}x(m) \sim \overline{M}_{\alpha}x(m)$.
\end{lemma}

\begin{lemma}\label{l3.4}\cite[Theorem 1.1]{de}
Let $A$ be a bounded set in $\mathbb{R} ^n$. For each $x\in A$, a closed cubic interval $Q(x)$ centered at $x$ is given. Then one can choose, from among the given intervals $\{Q(x)\}_{x\in A}$, a sequence $\{Q_k\}_{k\in \mathbb{N}}$ (possible finite) such that:
\begin{itemize}
\item[\rm(i)] The set $A$ is covered by the sequence, i.e., $A\subset \bigcup _kQ_k$;
\item[\rm(ii)] No point in $\mathbb{R} ^n$ is in more than $C_n$ (a number that only depends on $n$) cubes of the sequence $\{Q_k\}_{k\in \mathbb{N}}$, i.e., for every $x\in \mathbb{R} ^n$, $\sum_k{\mathcal{X} _{Q_k}(x) \leq C_n}$.
\end{itemize}
\end{lemma}

\begin{proof}[Proof of Theorem \ref{t3.2}]
Proof of necessity. By referencing the proof of \cite[Lemma 3.4]{hao2} in the case of $1<p<\frac{1}{\alpha}$, let us now verify that \cite[Lemma 3.4]{hao2} also holds true in the case of $p=1$. For every $\lambda>0$, $M\in\mathbb{Z_+}$, $N\in\mathbb{N}$ and $k\in\mathbb{Z}$, we define
\begin{equation*}
\begin{aligned}
& S_{k,N}=\mathbb{Z}\cap Q_{k,N},\enspace\mathrm{where}\enspace Q_{k,N}:=\{y\in \mathbb{R} :\left| y-k \right|\leq N\};\\
& E_{\lambda}:=\{k\in \mathbb{Z} :M_{\alpha}x(k)>\lambda \};\\
& E_{\lambda ,M}:=E_{\lambda}\cap S_{0,M},\enspace\mathrm{where}\enspace S_{0,M}:=\{m\in\mathbb{Z}:\left|m-0\right|\leq M\}.
\end{aligned}
\end{equation*}
Thus, for every $k\in E_{\lambda,M}$, by the definition of $M_\alpha$, there exists a $S_{k,N_k}$ such that
\begin{equation}\label{eq3.2}
\left|S_{k,N_k}\right|^{-1+\alpha}\sum_{m\in S_{k,N_k}}\left|x(m)\right|>\lambda.
\end{equation}
Since $E_{\lambda,M}\subset\mathop{\bigcup}_{k\in E_{\lambda,M}}S_{k,N_k}\subset\mathop{\bigcup}_{k\in E_{\lambda,M}}Q_{k,{N_k}}$, by Lemma \ref{l3.4}, there exists $\{k_j\}\subset E_{\lambda,M}$ such that $E_{\lambda,M}\subset \mathop{\bigcup}_j Q_{k_j,N_j}$ and $\sum_j\mathcal{X}_{Q_{k_j,N_j}}(k)\leq C_1$. Then we further have
\begin{equation}\label{eq3.3}
E_{\lambda,M}\subset\bigg(\mathop{\bigcup}\limits_j Q_{k_j,N_j}\bigg)\bigcap\mathbb{Z}=\mathop{\bigcup}\limits_j S_{k_j,N_j},
\end{equation}
\begin{equation}\label{eq3.4}
\sum\limits_j\mathcal{X}_{S_{k_j,N_j}}(k)\leq\sum\limits_j\mathcal{X}_{Q_{k_j,N_j}}(k)\leq C_1.
\end{equation}
From (\ref{eq3.2}), (\ref{eq3.3}), (\ref{eq3.4}), $\frac{1}{q}<1$, $\omega\in \mathcal{A}(1,q)$ and discrete H\"older's inequality, we conclude that
\begin{align*}
&\left(\sum_{m\in E_{\lambda,M}}\omega(m)^q\right)^\frac{1}{q}\leq \left(\sum\limits_j\sum\limits_{m\in S_{k_j,N_j}}\omega(m)^q\right)^\frac{1}{q}\leq\sum\limits_j\left(\sum\limits_{m\in S_{k_j,N_j}}\omega(m)^q\right)^\frac{1}{q}
\\
\leq& \sum_j\left(\sum_{m\in S_{k_j,N_j}}\omega(m)^q\right)^\frac{1}{q}\left(\frac{1}{\lambda\left|S_{k_j,N_j}\right|^{1-\alpha}}
\sum_{m\in S_{k_j,N_j}}\left|x(m)\right|\right)
\\
\leq& \sum_j{\left( \sum_{m\in S_{k_j,N_j}}{\omega \left( m \right) ^q} \right) ^{\small{\frac{1}{q}}}}\lambda ^{-1}\left| S_{k_j,N_j} \right|^{-\small{\frac{1}{q}}}\left( \sum_{m\in S_{k_j,N_j}}{\left| x\left( m \right) \right|\omega \left( m \right)} \right) \left( \underset{m\in S_{k_j,N_j}}{\mathrm{sup}}\frac{1}{\omega \left( m \right)} \right)
\\
\leq& \lambda ^{-1}\left\| \omega \right\| _{\mathcal{A} \left( 1,q \right) \left( \mathbb{Z} \right)}\sum_j{\sum_{m\in S_{k_j,N_j}}{\left| x\left( m \right) \right|\omega \left( m \right)}}\leq C_1 \lambda ^{-1}\left\| \omega \right\| _{\mathcal{A} \left( 1,q \right) \left( \mathbb{Z} \right)}\sum_{m\in \mathbb{Z}}{\left| x\left( m \right) \omega \left( m \right) \right|},
\end{align*}
and letting $M\rightarrow +\infty$ on both sides of above inequality, we obtain (\ref{eq3.1}).

Proof of sufficiency. If $M_{\alpha}$ is bounded from $l_{\omega}^{1}$ to $l_{\omega^q}^{q,weak}$, let us verify $\omega\in \mathcal{A}(1,q)$. Fix a bounded interval $S:=S_{m,N}\subset \mathbb{Z}$. Let $A:=\underset{k\in S}{\mathrm{inf}}\ \omega(k)$. Then $0<A<\infty$. Here, we only consider the case that $S$ is not a single point set. If $S$ is a single point set, the proof is obvious. Given $\epsilon >0$, there is a subset $E\subset S$ such that, for every $k\in E$,
\begin{equation}\label{eq3.5}
\omega\left( k \right) <A+\epsilon.
\end{equation}
Let $x(k) :=\mathcal{X} _E(k)$ and $2\lambda =|S|^{\alpha-1}\sum_{k\in S}{|x(k)|}$. For every $k\in S$, we have $\overline{M}_{\alpha}x(k)\geq 2\lambda >\lambda$ and hence $S\subset \{k\in \mathbb{Z} :\overline{M}_{\alpha}x(k)>\lambda \}$. By (\ref{eq3.1}) with $2\lambda =|S|^{\alpha-1}\sum_{k\in S}{|x(k)|}$, Lemma \ref{l3.3}, (\ref{eq3.5}) and $\frac{1}{q}=1-\alpha$, we obtain
\begin{align*}
\left( \sum_{k\in S}{\omega \left( k \right) ^q} \right) ^{\small{\frac{1}{q}}}&\leq \left( \sum_{\left\{ k\in \mathbb{Z} :\overline{M}_{\alpha}x(k) >\lambda \right\}}{\omega \left( k \right) ^q} \right) ^{\small{\frac{1}{q}}}\leq C\left| S \right|^{1-\alpha}\left| E \right|^{-1}\sum_{k\in E}{\omega \left( k \right)}
\\
&\leq C\left| S \right|^{1-\alpha}\left| E \right|^{-1}\left| E \right|\left( A+\epsilon \right) =C\left| S \right|^{\small{\frac{1}{q}}}\left( A+\epsilon \right).
\end{align*}
Hence by the arbitrariness of $\epsilon$, we obtain
\begin{equation*}
\left( \frac{1}{\left| S \right|}\sum_{k\in S}{\omega\left( k \right) ^q} \right) ^{\small{\frac{1}{q}}}\left( \underset{k\in S}{\mathrm{sup}}\ \frac{1}{\omega\left( k \right)} \right) \leq C,
\end{equation*}
which implies that $\omega\in \mathcal{A}(1,q)$. We finish the proof of Theorem \ref{t3.2}.
\end{proof}

The following Theorem \ref{t3.5} is another sufficient and necessary result for weight class of $\mathcal{A}(1,q)$.

\begin{theorem}\label{t3.5}
Let $0<\alpha <1$ and $\frac{1}{q}=1-\alpha$. Then $\omega\in \mathcal{A}(1,q)$ if and only if $I_{\alpha}$ is bounded from $l_{\omega}^{1}$ to $l_{\omega^q}^{q,weak}$, i.e., there exists a positive constant $C$ such that, for every $a>0$,
\begin{equation}\label{eq3.6}
\left( \sum_{\left\{ k\in \mathbb{Z} :\left| I_{\alpha}x\left( k \right) \right|>a \right\}}{\omega\left( k \right) ^q} \right) ^{\small{\frac{1}{q}}}\leq \frac{C}{a}\sum_{k\in \mathbb{Z}}{\left| x\left( k \right) \right| \omega\left( k \right)}.
\end{equation}
\end{theorem}

To prove Theorem \ref{t3.5}, we need the following lemma.
\begin{lemma}\label{l3.6}\cite[Lemma 4.9]{hao2}
Let $0<\alpha<1$ and $x=\{x(k)\}_{k\in\mathbb{Z}}$ be a sequence. If $|x|$ is a monotonic sequence on $\mathrm{supp}\,x$, then $M_\alpha x(k)\leq 2I_\alpha(|x|)(k)$ holds for every $k\in\mathbb{Z}$.
\end{lemma}

\begin{proof}[Proof of Theorem \ref{t3.5}]
Proof of necessity. Let $S:=S_{m,N}$. We may assume that $x=\{x(k) \}_{k\in \mathbb{Z}}$ is nonnegative and supported in some symmetric interval $S$, otherwise replace $x$ with $|x|=\{|x(k)|\mathcal{X}_{\{|k|\leq M \}}(k) \}_{k\in \mathbb{Z}}$, where $M\in \mathbb{Z} _+$.
From Proposition \ref{p2.3}(i) with $\omega\in \mathcal{A}(1,q)$, it follows that $\omega^q\in \mathcal{A}_1\subset \mathcal{A}_{\infty}$.

Moreover, let $a>0$ and $b\geq 6$. By checking the estimate of \cite[(4.6)]{hao2} with $\omega^q\in \mathcal{A}_{\infty}$, there exists a constant $c=c_{\alpha,\omega,b}>0$ such that
\begin{equation}\label{eq3.7}
\begin{aligned}
\omega^q\left( \left\{ k\in \mathbb{Z} :I_{\alpha}x\left( k \right) >ab \right\} \right)\leq& 2\omega^q\left( \left\{ k\in \mathbb{Z} :M_{\alpha}x\left( k \right) >ac \right\} \right)
\\
&+\frac{b^{-q}}{2}\omega^q\left( \left\{ k\in \mathbb{Z} :I_{\alpha}x\left( k \right) >a \right\} \cap 3S \right).
\end{aligned}
\end{equation}
Next, multiply both sides of (\ref{eq3.7}) by $a^q$ and take the supremum of both sides for $a\in (0,N)$, where $N>0$. We obtain
\begin{align*}
&\underset{0<a<N}{\mathrm{sup}}a^q\omega^q\left( \left\{ k\in \mathbb{Z} :I_{\alpha}x\left( k \right) >ab \right\} \right)
\\
\leq& \underset{0<a<N}{\mathrm{sup}}2a^q\omega^q\left( \left\{ k\in \mathbb{Z} :M_{\alpha}x\left( k \right) >ac \right\} \right)
+\underset{0<a<N}{\mathrm{sup}}\frac{a^qb^{-q}}{2}\omega^q\left( \left\{ k\in \mathbb{Z} :I_{\alpha}x\left( k \right) >a \right\} \cap 3S \right).
\end{align*}
By the method of substitution, we further obtain
\begin{align*}
&\underset{0<a<bN}{\mathrm{sup}}a^qb^{-q}\omega^q\left( \left\{ k\in \mathbb{Z} :I_{\alpha}x\left( k \right) >a \right\} \right)
\\
\leq &\underset{0<a<cN}{\mathrm{sup}}2c^{-q}a^q\omega^q\left( \left\{ k\in \mathbb{Z} :M_{\alpha}x\left( k \right) >a \right\} \right) +
\underset{0<a<bN}{\mathrm{sup}}\frac{a^qb^{-q}}{2}\omega^q\left( \left\{ k\in \mathbb{Z} :I_{\alpha}x\left( k \right) >a \right\} \right).
\end{align*}
Therefore,
\begin{equation}\label{eq3.8}
\underset{0<a<bN}{\mathrm{sup}}a^q\omega^q\left( \left\{ k\in \mathbb{Z} :I_{\alpha}x\left( k \right) >a \right\} \right) \leq 4b^qc^{-q}\underset{0<a<cN}{\mathrm{sup}}a^q\omega^q\left( \left\{ k\in \mathbb{Z} :M_{\alpha}x\left( k \right) >a \right\} \right).
\end{equation}
Letting $N\rightarrow \infty$ on both sides of (\ref{eq3.8}), it implies that
\begin{equation*}
\underset{a>0}{\mathrm{sup}}\ a^q\sum_{\left\{ k\in \mathbb{Z} :I_{\alpha}x\left( k \right) >a \right\}}{\omega\left( k \right)^q}\leq 4b^qc^{-q}\ \underset{a>0}{\mathrm{sup}}\ a^q\sum_{\left\{ k\in \mathbb{Z} :M_{\alpha}x\left( k \right) >a \right\}}{\omega\left( k \right)^q}.
\end{equation*}

If $x$ doesn't have compact support, applying $|x|=\{|x(k)|\mathcal{X}_{\{|k|\leq M \}}(k)\}_{k\in \mathbb{Z}}$, $M\in \mathbb{Z}_+$ to all the above processes,
\begin{equation*}
\underset{a>0}{\mathrm{sup}}\ a^q\sum_{\left\{ k\in \mathbb{Z} :I_{\alpha}\left( \left| x \right|\mathcal{X} _{\left\{ \left| \cdot \right|\leq M \right\}} \right) \left( k \right) >a \right\}}{\omega\left( k \right)^q}\leq 4b^qc^{-q}\ \underset{a>0}{\mathrm{sup}}\ a^q\sum_{\left\{ k\in \mathbb{Z} :M_{\alpha}\left( \left| x \right|\mathcal{X} _{\left\{ \left| \cdot \right|\leq M \right\}} \right) \left( k \right) >a \right\}}{\omega\left( k \right)^q}.
\end{equation*}
Taking $M\rightarrow \infty$ on above inequality and using monotone convergence theorem, we have
\begin{equation*}
\underset{a>0}{\mathrm{sup}}\ a^q\sum_{\left\{ k\in \mathbb{Z} : \left| I_{\alpha}x\left( k \right) \right|>a \right\}}{\omega\left( k \right)^q}\leq 4b^qc^{-q}\ \underset{a>0}{\mathrm{sup}}\ a^q\sum_{\left\{ k\in \mathbb{Z} : M_{\alpha}x\left( k \right) >a \right\}}{\omega\left( k \right)^q}.
\end{equation*}
Thus, from this and Theorem \ref{t3.2} with $\omega\in \mathcal{A}(1,q)$, we deduce that (\ref{eq3.6}) holds true.

Proof of sufficiency. Fix a bounded interval $S:=S_{m,N}\subset \mathbb{Z}$. Let $A:=\underset{k\in S}{\mathrm{inf}}\ \omega(k)$. Then $0<A<\infty$. Here, we only consider the case that $S$ is not a single point set. If $S$ is a single point set, the proof is obvious. Given $\epsilon >0$, there is a subset $E\subset S$ such that, for every $k\in E$, $\omega(k) <A+\epsilon$. Let $x(k) :=\mathcal{X} _E(k)$. Obviously, $|x|$ is a monotonic sequence on $\mathrm{supp}\,x$. From this, Lemma \ref{l3.6} and (\ref{eq3.6}), it follows that
\begin{equation*}
\left( \sum_{\left\{ k\in \mathbb{Z} :M_{\alpha}x\left( k \right) >a \right\}}{\omega \left( k \right) ^q} \right) ^{\frac{1}{q}}\le \left( \sum_{\left\{ k\in \mathbb{Z} :I_{\alpha}x\left( k \right) >\small{\frac{a}{2}} \right\}}{\omega \left( k \right) ^q} \right) ^{\frac{1}{q}}\leq \frac{C}{a}\sum_{k\in \mathbb{Z}}{\left| x\left( k \right) \right|\omega \left( k \right)}.
\end{equation*}
Then by this and repeating the proof of sufficiency of Theorem \ref{t3.2} with taking $2a=|S|^{\alpha -1}\sum_{k\in S}{|x(k)|}$, we finally obtain $\omega\in \mathcal{A}(1,q)$. We finish the proof of Theorem \ref{t3.5}.
\end{proof}

The following Theorem \ref{t3.7} is a sufficient and necessary result for weight class of $\mathcal{A}(p,\infty)$, which implies that $I_{\alpha}$ has a property resembling discrete bounded mean oscillation.
\begin{theorem}\label{t3.7}
Let $0<\alpha <1$ and $p=\frac{1}{\alpha}$. Then $\omega\in \mathcal{A}(p,\infty)$ if and only if for every symmetric interval $S_{m,N}\subset \mathbb{Z}$, there exists a positive constant $C$ such that
\begin{equation}\label{eq3.9}
\left( \underset{k\in S_{m,N}}{\mathrm{sup}}\,\omega \left( k \right) \right) \frac{1}{\left| S_{m,N} \right|}\sum_{k\in S_{m,N}}{\left| I_{\alpha}x\left( k \right) -\left( I_{\alpha}x \right) _{S_{m,N}} \right|}\leq C\left( \sum_{k\in \mathbb{Z}}{\left| x\left( k \right) \omega \left( k \right) \right|^p} \right) ^{\small{\frac{1}{p}}},
\end{equation}
where $( I_{\alpha}x )_{S_{m,N}}:=\frac{1}{|S_{m,N}|}\sum_{k\in S_{m,N}}{I_{\alpha}x(k)}$.
\end{theorem}

For Theorem \ref{t3.7}, if the weight $\omega$ is bounded, then we further have the following corollary.

\begin{corollary}\label{c3.8}
Let $0<\alpha <1$ and $p=\frac{1}{\alpha}$. If $\omega\in \mathcal{A}(p,\infty)$ and $\omega$ is bounded, then for every symmetric interval $S_{m,N}\subset \mathbb{Z}$, there exists a positive constant $C_{\alpha}$ such that
\begin{equation*}
\left\| I_{\alpha}x \right\| _{BMO\left( \mathbb{Z} \right)}:=\underset{m\in \mathbb{Z} ,N\in \mathbb{N}}{\mathrm{sup}}\,\frac{1}{\left| S_{m,N} \right|}\sum_{k\in S_{m,N}}{\left| I_{\alpha}x\left( k \right) -\left( I_{\alpha}x \right) _{S_{m,N}} \right|\leq C_{\alpha}\left\| \omega \right\| _{l^{\infty}}^{-1}\left\| x \right\| _{l_{\omega ^p}^{p}}}.
\end{equation*}
\end{corollary}

\begin{proof}[Proof of Theorem \ref{t3.7}]
Let $S:=S_{m,N}$. If $S$ is a single point set, the proof is obvious. So we only consider the case that $S$ is not a single point set.

Proof of necessity. Let $2S=S_{m,2N}$ and $T:=(2S)^c$. Then,
\begin{equation*}
I_{\alpha}x\left( k \right) =\sum_{i\in 2S\setminus \left\{ k \right\}}{\frac{x\left( i \right)}{\left| k-i \right|^{1-\alpha}}}+\sum_{i\in T\setminus \left\{ k \right\}}{\frac{x\left( i \right)}{\left| k-i \right|^{1-\alpha}}}.
\end{equation*}
Let $B:=\underset{k\in S}{\mathrm{sup}}\ \omega(k)$. The left side of (\ref{eq3.9}) with $I_{\alpha}x$ replaced by $\sum_{i\in 2S\setminus \{ k \}}{\frac{x(i)}{|k-i|^{1-\alpha}}}$. By interchanging the order of summation, discrete H\"older's inequality and $\omega\in \mathcal{A} (p,\infty)$, we obtain
\begin{align*}
&\frac{B}{\left| S \right|}\sum_{k\in S}{\bigg| \sum_{i\in 2S\setminus \left\{ k \right\}}{\frac{\left| x\left( i \right) \right|}{\left| k-i \right|^{1-\alpha}}}-\frac{1}{\left| S \right|}\sum_{j\in S}{\sum_{i\in 2S\setminus \left\{ j \right\}}{\frac{\left| x\left( i \right) \right|}{\left| j-i \right|^{1-\alpha}}}} \bigg|}
\\
\leq& \frac{B}{\left| S \right|}\sum_{k\in S}{\sum_{i\in 2S\setminus \left\{ k \right\}}{\frac{\left| x\left( i \right) \right|}{\left| k-i \right|^{1-\alpha}}}}+\frac{B}{\left| S \right|}\sum_{j\in S}{\sum_{i\in 2S\setminus \left\{ j \right\}}{\frac{\left| x\left( i \right) \right|}{\left| j-i \right|^{1-\alpha}}}}
\\
=&\frac{2B}{\left| S \right|}\sum_{k\in S}{\sum_{i\in 2S\setminus \left\{ k \right\}}{\frac{\left| x\left( i \right) \right|}{\left| k-i \right|^{1-\alpha}}}}
=\frac{2B}{\left| S \right|}\sum_{i\in 2S}{\left| x\left( i \right) \right|\sum_{k\in S\setminus \left\{ i \right\}}{\frac{1}{\left| k-i \right|^{1-\alpha}}}}
\\
\leq &\frac{2B}{\left| S \right|}\sum_{i\in 2S}{\left| x\left( i \right) \right|\sum_{k\in S_{i,3N}\setminus \left\{ i \right\}}{\frac{1}{\left| k-i \right|^{1-\alpha}}}}
\\
\leq& \frac{4B}{\left| S \right|}\sum_{j=1}^{3N}{j^{\alpha -1}}\sum_{i\in 2S}{\left| x\left( i \right) \right|}
\leq \frac{4B}{\left| S \right|}\int_0^{3N}{y^{\alpha -1}}dy\sum_{i\in 2S}{\left| x\left( i \right) \right|}
\\
=&\frac{4B}{\alpha \left| S \right|}\left( 3N \right) ^{\alpha}\sum_{i\in 2S}{\left| x\left( i \right) \right|}
\leq C_{\alpha}B\left| S \right|^{-1+\alpha}\sum_{i\in 2S}{\left| x\left( i \right) \right|}
\\
\leq &C_{\alpha}B\left| S \right|^{-\small{\frac{1}{\boldsymbol{p}^{\prime}}}}\left( \sum_{i\in 2S}{\left| x\left( i \right) \omega\left( i \right) \right|^p} \right) ^{\small{\frac{1}{p}}}\left( \sum_{i\in 2S}{\omega\left( i \right) ^{-p^{\prime}}} \right) ^{\small{\frac{1}{p^{\prime}}}}
\\
\leq &C_{\alpha}\left( \frac{\left| 2S \right|}{\left| S \right|} \right) ^{\frac{1}{p^{\prime}}}\left( \sum_{i\in 2S}{\left| x\left( i \right) \omega \left( i \right) \right|^p} \right) ^{\small{\frac{1}{p}}}\left( \underset{i\in 2S}{\mathrm{sup}}\,\,\omega \left( k \right) \right) \left( \frac{1}{\left| 2S \right|}\sum_{i\in 2S}{\omega \left( i \right) ^{-p^{\prime}}} \right) ^{\small{\frac{1}{p^{\prime}}}}
\\
\leq &C_{\alpha}\left\| \omega \right\| _{\mathcal{A} \left( p,\infty \right) \left( \mathbb{Z} \right)}\left( \sum_{i\in 2S}{\left| x\left( i \right) \omega \left( i \right) \right|^p} \right) ^{\small{\frac{1}{p}}}\leq C_{\alpha}\left\| \omega \right\| _{\mathcal{A} \left( p,\infty \right) \left( \mathbb{Z} \right)}\left( \sum_{i\in \mathbb{Z}}{\left| x\left( i \right) \omega \left( i \right) \right|^p} \right) ^{\small{\frac{1}{p}}}.
\end{align*}

For every $k,j\in S$ and $i\in T$, we have
\begin{equation}\label{eq3.10}
\left| \left| k-i \right|^{\alpha -1}-\left| j-i \right|^{\alpha -1} \right| \leq C\left|S \right|\left|k-i \right|^{\alpha -2},
\end{equation}
which can be deduced by the continuous case in \cite[Page 270]{1974m}. The left side of (\ref{eq3.9}) with $I_{\alpha}x$ replaced by $\sum_{i\in T\setminus \{k\}}{\frac{x(i)}{|k-i|^{1-\alpha}}}$. By interchanging the order of summation, (\ref{eq3.10}) and discrete H\"older's inequality, we obtain
\begin{align*}
&\frac{B}{\left| S \right|}\sum_{k\in S}{\bigg| \sum_{i\in T}{\frac{x\left( i \right)}{\left| k-i \right|^{1-\alpha}}-\frac{1}{\left| S \right|}\sum_{j\in S}{\sum_{i\in T}{\frac{x\left( i \right)}{\left| j-i \right|^{1-\alpha}}}}} \bigg|}
\\
=&\frac{B}{\left| S \right|}\sum_{k\in S}{\bigg| \frac{1}{\left| S \right|}\sum_{j\in S}{\sum_{i\in T}{\frac{x\left( i \right)}{\left| k-i \right|^{1-\alpha}}}-\frac{1}{\left| S \right|}\sum_{j\in S}{\sum_{i\in T}{\frac{x\left( i \right)}{\left| j-i \right|^{1-\alpha}}}}} \bigg|}
\\
\leq &\frac{B}{\left| S \right|}\sum_{k\in S}{\frac{1}{\left| S \right|}}\sum_{j\in S}{\sum_{i\in T}{\left| x\left( i \right) \right|\left| \left| k-i \right|^{\alpha -1}-\left| j-i \right|^{\alpha -1} \right|}}
\\
\leq &CB\sum_{i\in T}{\left| x\left( i \right) \right|}\sum_{k\in S}{\left| k-i \right|^{\alpha -2}}
\leq CB\left| S \right|\sum_{i\in T}{\left| x\left( i \right) \right|\left| m-i \right|^{\alpha -2}}
\\
\leq &CB\left| S \right|\left( \sum_{i\in T}{\left| x\left( i \right) \omega\left( i \right) \right|}^p \right) ^{\small{\frac{1}{p}}}\left( \sum_{i\in T}{\left| m-i \right|^{\left( \alpha -2 \right) p^{\prime}}\omega\left( i \right) ^{-p^{\prime}}} \right) ^{\small{\frac{1}{p^{\prime}}}}.
\end{align*}

Next, let us show that $B|S|(\sum_{i\in T}{|m-i|^{(\alpha -2) p^{\prime}}\omega(i) ^{-p^{\prime}}}) ^{\small{\frac{1}{p^{\prime}}}}$ is bounded and independent of $S$. Let $S_k:=S_{m,2^kN}$, $k\in \mathbb{Z} _+$. Then $T=\bigcup_{k=2}^{\infty}{(S_k\setminus S_{k-1})}$. We have
\begin{align*}
&B\left| S \right|\left( \sum_{i\in T}{\left| m-i \right|^{\left( \alpha -2 \right) p^{\prime}}\omega\left( i \right) ^{-p^{\prime}}} \right) ^{\small{\frac{1}{p^{\prime}}}} =B\left| S \right|\left( \sum_{k=2}^{\infty}{\sum_{i\in S_k\setminus S_{k-1}}{\left| m-i \right|^{\left( \alpha -2 \right) p^{\prime}}\omega\left( i \right) ^{-p^{\prime}}}} \right) ^{\small{\frac{1}{p^{\prime}}}}
\\
\leq& B\left| S \right|\sum_{k=2}^{\infty}{\left( \sum_{i\in S_k\setminus S_{k-1}}{\left| m-i \right|^{\left( \alpha -2 \right) p^{\prime}}\omega\left( i \right) ^{-p^{\prime}}} \right) ^{\small{\frac{1}{p^{\prime}}}}}
\\
\leq &C_{\alpha}B\sum_{k=2}^{\infty}{\left( \sum_{i\in S_k\setminus S_{k-1}}{\left( 2^k\left| S \right| \right) ^{\left( \alpha -2 \right) p^{\prime}}\left| S \right|^{p^{\prime}}\omega\left( i \right) ^{-p^{\prime}}} \right) ^{\small{\frac{1}{p^{\prime}}}}}
\\
\leq &C_{\alpha}B\sum_{k=2}^{\infty}{\left( \sum_{i\in S_k}{\left( 2^k \right) ^{-p^{\prime}}\left( 2^k \right) ^{-1}\left| S \right|^{-1}\omega\left( i \right) ^{-p^{\prime}}} \right) ^{\small{\frac{1}{p^{\prime}}}}}
\\
=&C_{\alpha}B\sum_{k=2}^{\infty}{2^{-k}}\left( \sum_{i\in S_k}{\frac{1}{2^k\left| S \right|}\omega\left( i \right) ^{-p^{\prime}}} \right) ^{\small{\frac{1}{p^{\prime}}}}
\\
<&C_{\alpha}\sum_{k=2}^{\infty}{2^{-k}}B\left( \sum_{i\in S_k}{\frac{1}{\left| S_k \right|}\omega\left( i \right) ^{-p^{\prime}}} \right) ^{\frac{1}{p^{\prime}}} \leq C_{\alpha}\left\| \omega \right\| _{\mathcal{A} \left( p,\infty \right) \left( \mathbb{Z} \right)},
\end{align*}
then $B|S|(\sum_{i\in T}{|m-i|^{(\alpha -2) p^{\prime}}\omega(i) ^{-p^{\prime}}}) ^{\small{\frac{1}{p^{\prime}}}}$ is bounded and independent of $S$.

Proof of sufficiency. Let us prove this by two steps.

Step 1: Let us prove that there is an integer $t>1$ sufficiently large such that, for every bounded interval $S$ and every $i\in S$,
\begin{equation}\label{eq3.11}
\sum_{k\in tS\setminus \left\{ i \right\}}{\left| k-i \right|^{\alpha -1}}< \frac{t}{4}\sum_{k\in S\setminus \left\{ i \right\}}{\left| k-i \right|^{\alpha -1}}.
\end{equation}
To prove the existence of $t$, the range of summation on the left side of (\ref{eq3.11}) can be enlarged to center on $i$ and the interval is twice as long as $tS$, i.e.,
\begin{align}\label{eq3.12}
\sum_{k\in tS\setminus \left\{ i \right\}}{\left| k-i \right|^{\alpha -1}}
&\leq \sum_{k\in S_{i,2tN}\setminus \left\{ i \right\}}{\left| k-i \right|^{\alpha -1}}
=2\sum_{j=1}^{2tN}{j^{\alpha -1}}\nonumber
\\
&\leq 2\int_0^{2tN}{y^{\alpha -1}}dy
=\frac{2}{\alpha} \left( 2tN \right) ^{\alpha}.
\end{align}
Similarly, the range of summation on the right side of (\ref{eq3.11}) can be reduced, i.e.,
\begin{align}\label{eq3.13}
\sum_{k\in S\setminus \left\{ i \right\}}{\left| k-i \right|^{\alpha -1}}
&\geq \sum_{k\in \left[ i-N,i-1 \right] \cap \mathbb{Z}}{\left| k-i \right|^{\alpha -1}}=\sum_{j=1}^N{j^{\alpha -1}}\nonumber
\\
&\geq \int_1^{N+1}{y^{\alpha -1}}dy
=\frac{1}{\alpha}\left[ \left( N+1 \right) ^{\alpha}-1 \right].
\end{align}
Notice that
\begin{equation*}
\frac{\frac{2}{\alpha}\left( 2tN \right) ^{\alpha}}{\frac{1}{\alpha}\left[ \left( N+1 \right) ^{\alpha}-1 \right]}=2^{2+\alpha}t^{\alpha}\frac{\frac{1}{2}N^{\alpha}}{\left( N+1 \right) ^{\alpha}-1}\enspace\mathrm{and}\enspace\lim_{N\rightarrow +\infty} \frac{\frac{1}{2}N^{\alpha}}{\left( N+1 \right) ^{\alpha}-1}=\frac{1}{2}<1,
\end{equation*}
then there exists sufficiently large $N_0\in \mathbb{N}$ such that, for any $N>N_0$, $\frac{\frac{1}{2}{N_0}^{\alpha}}{(N_0+1) ^{\alpha}-1}<1$. Therefore, for any $N>N_0$ and sufficiently large $t$, by (\ref{eq3.12}), (\ref{eq3.13}) and $0<\alpha<1$, we obtain
\begin{equation*}
\frac{\sum_{k\in tS\setminus \left\{ i \right\}}{\left| k-i \right|^{\alpha -1}}}{\sum_{k\in S\setminus \left\{ i \right\}}{\left| k-i \right|^{\alpha -1}}}\leq \frac{\frac{2}{\alpha}\left( 2tN \right) ^{\alpha}}{\frac{1}{\alpha}\left[ \left( N+1 \right) ^{\alpha}-1 \right]}<2^{2+\alpha}t^{\alpha}<\frac{t}{4},
\end{equation*}
which implies (\ref{eq3.11}).

Step 2: Let us prove $\omega\in \mathcal{A}(p,\infty)$. Fix a bounded interval $S$, let $\{x(k)\}_{k\in \mathbb{Z}}$ be a positive sequence on $S$ and 0 off $S$, and pick integer $t>1$ as in (\ref{eq3.11}). By interchanging the order of summation and (\ref{eq3.11}), we have
\begin{align*}
\left( I_{\alpha}x \right) _S-\left( I_{\alpha}x \right) _{tS}
&=\frac{1}{\left| S \right|}\sum_{k\in S}{\sum_{i\in S\setminus \left\{ k \right\}}{\frac{x\left( i \right)}{\left| k-i \right|^{1-\alpha}}}}-\frac{1}{\left| tS \right|}\sum_{k\in tS}{\sum_{i\in S\setminus \left\{ k \right\}}{\frac{x\left( i \right)}{\left| k-i \right|^{1-\alpha}}}}
\\
&=\frac{1}{\left| S \right|}\sum_{i\in S}{x\left( i \right) \bigg( \sum_{k\in S\setminus \left\{ i \right\}}{\frac{1}{\left| k-i \right|^{1-\alpha}}-\frac{\left| S \right|}{\left| tS \right|}\sum_{k\in tS\setminus \left\{ i \right\}}{\frac{1}{\left| k-i \right|^{1-\alpha}}}} \bigg)}
\\
&\geq \frac{1}{\left| S \right|}\sum_{i\in S}{x\left( i \right) \bigg( \sum_{k\in S\setminus \left\{ i \right\}}{\frac{1}{\left| k-i \right|^{1-\alpha}}-\frac{2}{t}\sum_{k\in tS\setminus \left\{ i \right\}}{\frac{1}{\left| k-i \right|^{1-\alpha}}}} \bigg)}
\\
&> \frac{1}{2\left| S \right|}\sum_{i\in S}{x\left( i \right)}\sum_{k\in S\setminus \left\{ i \right\}}{\frac{1}{\left| k-i \right|^{1-\alpha}}}
\\
&=\frac{1}{2\left| S \right|}\sum_{k\in S}{\sum_{i\in S\setminus \left\{ k \right\}}{\frac{x\left( i \right)}{\left| k-i \right|^{1-\alpha}}}}
=\frac{1}{2}\left( I_{\alpha}x \right) _S.
\end{align*}
From this and (\ref{eq3.9}), it follows that
\begin{align}\label{eq3.14}
\left( I_{\alpha}x \right) _S&< 2\left[ \left( I_{\alpha}x \right) _S-\left( I_{\alpha}x \right) _{tS} \right]
\leq \frac{2}{\left| S \right|}\sum_{k\in S}{\left| \left( I_{\alpha}x \right) _S-\left( I_{\alpha}x \right) _{tS} \right|}\nonumber
\\
&\leq \frac{2}{\left| S \right|}\sum_{k\in S}{\left| \left( I_{\alpha}x \right) _S-I_{\alpha}x\left( k \right) \right|}+\frac{2}{\left| S \right|}\sum_{k\in S}{\left| I_{\alpha}x\left( k \right) -\left( I_{\alpha}x \right) _{tS} \right|}\nonumber
\\
&\leq \frac{2}{\left| S \right|}\sum_{k\in S}{\left| \left( I_{\alpha}x \right) _S-I_{\alpha}x\left( k \right) \right|}+\frac{2\left| tS \right|}{\left| S \right|}\frac{1}{\left| tS \right|}\sum_{k\in tS}{\left| I_{\alpha}x\left( k \right) -\left( I_{\alpha}x \right) _{tS} \right|}\nonumber
\\
&\leq C\left( 2+4t \right) \left( \underset{k\in S}{\mathrm{sup}}\ \omega\left( k \right) \right) ^{-1}\left( \sum_{k\in \mathbb{Z}}{\left| x\left( k \right) \omega\left( k \right) \right|^p} \right) ^{\small{\frac{1}{p}}}.
\end{align}
By a simple calculation, we have
\begin{align*}
\frac{1}{\left| S \right|}\sum_{k\in S}{\sum_{i\in S}{\frac{\omega\left( i \right) ^{-p^{\prime}}}{\left| S \right|^{1-\alpha}}}}&=\frac{1}{\left| S \right|}\sum_{k\in S}{\sum_{i\in S\setminus \left\{ k \right\}}{\frac{\omega\left( i \right) ^{-p^{\prime}}}{\left| S \right|^{1-\alpha}}}}+\frac{1}{\left| S \right|}\frac{1}{\left| S \right|}\sum_{k\in S}{\sum_{i\in S}{\frac{\omega\left( i \right) ^{-p^{\prime}}}{\left| S \right|^{1-\alpha}}}}
\\
\Longrightarrow \frac{1}{\left| S \right|}\sum_{k\in S}{\sum_{i\in S}{\frac{\omega\left( i \right) ^{-p^{\prime}}}{\left| S \right|^{1-\alpha}}}}&=\frac{1}{\left| S \right|-1}\sum_{k\in S}{\sum_{i\in S\setminus \left\{ k \right\}}{\frac{\omega\left( i \right) ^{-p^{\prime}}}{\left| S \right|^{1-\alpha}}}}.
\end{align*}
Pick $x(i) =\omega(i) ^{-p^{\prime}}\mathcal{X} _S(i)$. By above equality and (\ref{eq3.14}), we obtain
\begin{align*}
&\left| S \right|^{\alpha -1}\sum_{i\in S}{\omega\left( i \right) ^{-p^{\prime}}}=\frac{1}{\left| S \right|}\sum_{k\in S}{\sum_{i\in S}{\frac{\omega\left( i \right) ^{-p^{\prime}}}{\left| S \right|^{1-\alpha}}}}
\\
=&\frac{\left| S \right|}{\left| S \right|-1}\frac{1}{\left| S \right|}\sum_{k\in S}{\sum_{i\in S\setminus \left\{ k \right\}}{\frac{\omega\left( i \right) ^{-p^{\prime}}}{\left| S \right|^{1-\alpha}}}}<\frac{\left| S \right|}{\left| S \right|-1}\frac{1}{\left| S \right|}\sum_{k\in S}{\sum_{i\in S\setminus \left\{ k \right\}}{\frac{\omega\left( i \right) ^{-p^{\prime}}}{\left| k-i \right|^{1-\alpha}}}}
\\
<& \frac{2}{\left| S \right|}\sum_{k\in S}{\sum_{i\in S\setminus \left\{ k \right\}}{\frac{\omega\left( i \right) ^{-p^{\prime}}\mathcal{X} _S\left( i \right)}{\left| k-i \right|^{1-\alpha}}}}=\frac{2}{\left| S \right|}\sum_{k\in S}{\sum_{i\in S\setminus \left\{ k \right\}}{\frac{x\left( i \right)}{\left| k-i \right|^{1-\alpha}}}}=2\left( I_{\alpha}x \right) _S
\\
<& C_{t}\left( \underset{k\in S}{\mathrm{sup}}\ \omega\left( k \right) \right) ^{-1}\left( \sum_{k\in S}{\left| w\left( k \right) ^{-p^{\prime}}\omega\left( k \right) \right|^p} \right) ^{\small{\frac{1}{p}}}.
\end{align*}
By $0<\alpha <1$, $p=\frac{1}{\alpha}$, $p^{\prime}=\frac{1}{1-\alpha}$ and simple calculation, we immediately obtain $\omega\in \mathcal{A} (p,\infty)$. We finish the proof of Theorem \ref{t3.7}.
\end{proof}

\section{Applications \label{s4}}
In this section, we prove the boundedness of discrete fractional maximal operator and discrete Riesz potential on discrete weighted Lebesgue spaces. We also give some applications for the cases $\omega\in \mathcal{A}(1,q)$ and $\omega\in \mathcal{A}(p,\infty)$.

The following Theorem \ref{t4.1} is a supplement of Theorem \ref{t3.2}.

\begin{theorem}\label{t4.1}
Let $0<\alpha<1$, $q_0>1$ satisfying $\frac{1}{q_0}=1-\alpha$ and $\omega\in \mathcal{A}(1,q_0)$. Then there exist a positive constant $C$, $q_1>q_0$ and $1<p_1<\frac{1}{\alpha}$ such that $\frac{1}{q_1}=\frac{1}{p_1}-\alpha$ and
\begin{equation}\label{eq4.1}
\left\| M_{\alpha}x \right\| _{l_{\omega ^q}^{q}}\leq C\left\| x \right\| _{l_{\omega ^p}^{p}},\enspace t\in \left[ 0,1 \right),\enspace \frac{1}{p}=t+\frac{1-t}{p_1},\enspace \frac{1}{q}=\frac{t}{q_0}+\frac{1-t}{q_1}.
\end{equation}
\end{theorem}

\begin{proof}
If $\omega \in \mathcal{A}(1,q_0)$, then by Proposition \ref{p2.3}(iv), there exist $p_1\in(1,\frac{1}{\alpha})$ and $q_1\in(q_0,\infty)$ such that $\frac{1}{q_1}=\frac{1}{p_1}-\alpha$ and $\omega \in \mathcal{A}(p,q)$, where $\frac{1}{p}=t+\frac{1-t}{p_1}$, $\frac{1}{q}=\frac{t}{q_0}+\frac{1-t}{q_1}$, $t\in[0,1)$. From this and Theorem \ref{t1.3}, (\ref{eq4.1}) holds true and hence we finish the proof of Theorem \ref{t4.1}.
\end{proof}

Theorem \ref{t4.1} seems new even in continuous setting, which is the following Corollary \ref{c4.2}.

\begin{corollary}\label{c4.2}
Let $0<\alpha<n$, $q_0>1$ satisfying $\frac{1}{q_0}=1-\frac{\alpha}{n}$ and $\omega\in \mathcal{A}(1,q_0)$, i.e., for any cube $Q$ in $\mathbb{R}^n$,
\begin{equation*}
\underset{Q}{\mathrm{sup}}\left( \frac{1}{|Q|}\int_Q{\omega}(x)^{q_0}dx \right) ^{\frac{1}{q_0}}\left( \underset{x\in Q}{\mathrm{ess\ sup}}\,\,\frac{1}{\omega \left( x \right)} \right) <\infty.
\end{equation*}
Then there exist a positive constant $C$, $q_1>q_0$ and $1<p_1<\frac{n}{\alpha}$ such that $\frac{1}{q_1}=\frac{1}{p_1}-\frac{\alpha}{n}$ and
\begin{equation*}
\left\| \mathcal{M}_{\alpha}f \right\| _{L_{\omega ^q}^{q}}\leq C\left\| f \right\| _{L_{\omega ^p}^{p}},\enspace t\in \left[ 0,1 \right),\enspace \frac{1}{p}=t+\frac{1-t}{p_1},\enspace \frac{1}{q}=\frac{t}{q_0}+\frac{1-t}{q_1},
\end{equation*}
where $\mathcal{M}_{\alpha}f(x) :=\underset{Q}{\mathrm{sup}}\ |Q|^{-1+{{\alpha}/{n}}}\int_Q{|f(y)|dy}$.
\end{corollary}

\begin{proof}
This proof is similar to that of the discrete version (see the proof of Theorem \ref{t4.1}). The details are omitted.
\end{proof}

When $\omega\in \mathcal{A}(1,q)$ and $x\in l_{\omega}^{1}$, then by Theorem \ref{t3.2}, $M_{\alpha}x\in l_{\omega ^q}^{q,weak}$. By the following Theorem \ref{t4.3}, we find a sufficient condition for $M_{\alpha}x\in l_{\omega ^q}^{q,weak}(S_{m,N})$, where sequence $x$ is supported in some bounded symmetric interval $S_{m,N}$ and satisfies some sum condition of log-type which is a discrete variant of \cite[Proposition 1]{welland}. Similar result also holds for discrete Riesz potential, and it is also a discrete variant of \cite[Page 19, (2.2)]{welland}.

\begin{theorem}\label{t4.3}
Let $0<\alpha <1$, $\frac{1}{q}=1-\alpha$, $\omega\in \mathcal{A}(1,q)$ and $v=\omega ^q$. If $x$ is supported in a symmetric interval $S_{m,N}\subset \mathbb{Z}$ and $\sum_{k\in S_{m,N}}{|x(k)| \omega(k) \log ^+|x(k) \omega(k)^{1-q}|}<\infty$. Then there exists a positive constant $C$ such that
\begin{small}
\begin{equation}\label{eq4.2}
\Bigg( \sum_{k\in S_{m,N}}{\left| M_{\alpha}x\left( k \right) \omega\left( k \right) \right|^q} \Bigg) ^{\small{\frac{1}{q}}}\leq C\Bigg( v\left( S_{m,N} \right) +\sum_{k\in S_{m,N}}{\left|x\left( k \right) \right| \omega\left( k \right) \log ^+\left| x\left( k \right) \omega\left( k \right) ^{1-q} \right|} \Bigg)
\end{equation}
\end{small}
and
\begin{small}
\begin{equation}\label{eq4.3}
\Bigg( \sum_{k\in S_{m,N}}{\left| I_{\alpha}x\left( k \right) \omega\left( k \right) \right|^q} \Bigg) ^{\small{\frac{1}{q}}}\leq C\Bigg( v\left( S_{m,N} \right) +\sum_{k\in S_{m,N}}{\left|x\left( k \right) \right| \omega\left( k \right) \log ^+\left| x\left( k \right) \omega\left( k \right) ^{1-q} \right|} \Bigg),
\end{equation}
\end{small}
where
\begin{equation*}
\log ^+t:=\begin{cases}
	\log t, \qquad t>1\\
	~~0~~, \quad 0\leq t\leq 1\\
\end{cases}.
\end{equation*}
\end{theorem}

\begin{proof}
Since $\omega\in \mathcal{A}(1,q)$, by Proposition \ref{p2.3}(iv), there exist $p_1\in(1,\frac{1}{\alpha})$ and $q_1\in(q,\infty)$ such that $\frac{1}{q_1}=\frac{1}{p_1}-\alpha$, $\omega \in \mathcal{A}(p_1,q_1)$ and $\omega \in \mathcal{A}(p_t,q_t)$, where $\frac{1}{p_t}=t+\frac{1-t}{p_1}$, $\frac{1}{q_t}=\frac{t}{q}+\frac{1-t}{q_1}$, $t\in[0,1)$. By Proposition \ref{p2.3}(i), we have $\omega \in \mathcal{A}(1,q)\Longleftrightarrow \omega ^q\in \mathcal{A} _1$, then $\omega ^q\in \mathcal{A} _{1+\frac{q_t}{p_t\prime}}$ and $\omega ^{\frac{q}{q_t}}\in \mathcal{A} (p_t,q_t)$. From this and Theorem \ref{t4.1}, it follows that
\begin{equation}\label{eq4.4}
\left( \sum_{k\in \mathbb{Z}}{\left| M_{\alpha}x\left( k \right) \omega \left( k \right) ^{\frac{q}{q_t}} \right|^{q_t}} \right) ^{\frac{1}{q_t}}\leq C\left( \sum_{k\in \mathbb{Z}}{\left| x\left( k \right) \omega \left( k \right) ^{\frac{q}{q_t}} \right|^{p_t}} \right) ^{\frac{1}{p_t}}.
\end{equation}
Define a sublinear operator $T$ by $Tg(k):=M_{\alpha}(g\omega ^{q\alpha})(k)$, where $g(k):=x(k)\omega (k)^{-q\alpha}$. Then (\ref{eq4.4}) can be written as
\begin{equation}\label{eq4.5}
\left( \sum_{k\in \mathbb{Z}}{\left| Tg\left( k \right) \right|^{q_t}v\left( k \right)} \right) ^{\frac{1}{q_t}}\le C\left( \sum_{k\in \mathbb{Z}}{\left| g\left( k \right) \right|^{p_t}}v\left( k \right) \right) ^{\frac{1}{p_t}}.
\end{equation}
Let $S:=S_{m,N}$, $E:=\{k\in S:|g(k)|\leq 1 \}$, $E_n:=\{k\in S:2^n<|g(k)|\leq 2^{n+1} \}$, $n\in \mathbb{N}$ and $g_n(k):=g(k) \mathcal{X} _{E_n}(k) $. Let $t_n:=\frac{n+1}{n+2}$, $\frac{1}{p_n}=t_n+\frac{1-t_n}{p_1}$, $\frac{1}{q_n}=\frac{t_n}{q}+\frac{1-t_n}{q_1}$ and $\frac{1}{q_n}=\frac{1}{p_n}-\alpha$. Particularly, when $|g(k)|\leq 1$, we choose $\frac{1}{\widetilde{p}}=\frac{1}{2}+\frac{1}{2p_1}$ and $\frac{1}{\widetilde{q}}=\frac{1}{2q}+\frac{1}{2q_1}$. Then by Minkowski's inequality, discrete H\"older's inequality and (\ref{eq4.5}), we have
\begin{align}\label{eq4.6}
&\left( \sum_{k\in S}{\left| Tg\left( k \right) \right|^qv\left( k \right)} \right) ^{\frac{1}{q}}\nonumber
\\
\leq& \left( \sum_{k\in S}{\left| T\left( g\mathcal{X} _E \right) \left( k \right) \right|^qv\left( k \right)} \right) ^{\frac{1}{q}}+\sum_{n=0}^{\infty}{\left( \sum_{k\in S}{\left| Tg_n\left( k \right) \right|^qv\left( k \right)} \right) ^{\frac{1}{q}}}\nonumber
\\
\leq& \left\{ \left[ \sum_{k\in S}{\left( \left| T\left( g\mathcal{X} _E \right) \left( k \right) \right|^qv\left( k \right) ^{\frac{q}{\widetilde{q}}} \right) ^{\frac{\widetilde{q}}{q}}} \right] ^{\frac{q}{\widetilde{q}}}v\left( S \right) ^{^{1-\frac{q}{\widetilde{q}}}} \right\} ^{\frac{1}{q}}\nonumber
\\
&+\sum_{n=0}^{\infty}{\left\{ \left[ \sum_{k\in S}{\left( \left| Tg_n\left( k \right) \right|^qv\left( k \right) ^{\frac{q}{q_n}} \right) ^{\frac{q_n}{q}}} \right] ^{\frac{q}{q_n}}v\left( S \right) ^{1-\frac{q}{q_n}} \right\}}^{\frac{1}{q}}\nonumber
\\
=&v\left( S \right) ^{\frac{1}{q}-\frac{1}{\widetilde{q}}}\left[ \sum_{k\in S}{\left( \left| T\left( g\mathcal{X} _E \right) \left( k \right) \right|^{\widetilde{q}}v\left( k \right) \right)} \right] ^{\frac{1}{\widetilde{q}}}+\sum_{n=0}^{\infty}{v\left( S \right) ^{\frac{1}{q}-\frac{1}{q_n}}}\left( \sum_{k\in S}{\left| Tg_n\left( k \right) \right|^{q_n}v\left( k \right)} \right) ^{\frac{1}{q_n}}\nonumber
\\
\leq& Cv\left( S \right) ^{1-\frac{1}{\widetilde{p}}}\left( \sum_{k\in S}{\left| g\left( k \right) \mathcal{X} _E\left( k \right) \right|^{\widetilde{p}}v\left( k \right)} \right) ^{\frac{1}{\widetilde{p}}}+\sum_{n=0}^{\infty}{Cv\left( S \right) ^{1-\frac{1}{p_n}}}\left( \sum_{k\in S}{\left| g_n\left( k \right) \right|^{p_n}v\left( k \right)} \right) ^{\frac{1}{p_n}}\nonumber
\\
\leq& Cv\left( S \right) ^{1-\frac{1}{\widetilde{p}}}\left( \sum_{k\in S}{v\left( k \right)} \right) ^{\frac{1}{\widetilde{p}}}+Cv\left( S \right) \sum_{n=0}^{\infty}{\left( n+2 \right) 2^{n+1}v\left( S \right) ^{-\frac{1}{p_n}}\left( \sum_{k\in E_n}{v\left( k \right)} \right) ^{\frac{1}{p_n}}}\nonumber
\\
=&Cv\left( S \right)+Cv\left( S \right) \sum_{n=0}^{\infty}{\left( n+2 \right) 2^{n+1}\left( \frac{v\left( E_n \right)}{v\left( S \right)} \right)}^{\frac{1}{p_n}}.
\end{align}

Let $K:=\{n\in \mathbb{N} :( \frac{v(E_n)}{v(S)}) ^{\frac{1}{p_n}}\leq 3^{-(n+1)} \}$ and $K^{\prime}:=\mathbb{N}\setminus K$. For $n \in K^{\prime}$, noting $\frac{v(E_n)}{v(S)}>3^{-(n+1) p_n}$, $\frac{1}{p_n}=t_n+\frac{1-t_n}{p_1}$ and $t_n=\frac{n+1}{n+2}$, it follows that $p_n-1=\frac{p_1-1}{p_1(n+1)+1}$, then
\begin{equation}\label{eq4.7}
\left( \frac{v\left( E_n \right)}{v\left( S \right)} \right) ^{\frac{1}{p_n}-1}<3^{\left( n+1 \right) \left( p_n-1 \right)}<3^{\left( n+1 \right) \frac{p_1-1}{p_1\left( n+1 \right)}}=3^{\frac{p_1-1}{p_1}}<3.
\end{equation}
Therefore, by (\ref{eq4.6}) and (\ref{eq4.7}), we obtain
\begin{align*}
&\left( \sum_{k\in S}{\left| Tg\left( k \right) \right|^qv\left( k \right)} \right) ^{\frac{1}{q}}
\\
\leq& Cv\left( S \right) +Cv\left( S \right) \sum_{n\in K}{\left( n+2 \right) \left( \frac{2}{3} \right)}^{n+1}+Cv\left( S \right) \sum_{n\in K^{\prime}}{3\left( n+2 \right) 2^{n+1}\frac{v\left( E_n \right)}{v\left( S \right)}}
\\
\leq& Cv\left( S \right) +C\sum_{n\in K^{\prime}}{\left( n+2 \right) 2^n}v\left( E_n \right)
\leq Cv\left( S \right) +C\sum_{n\in K^{\prime}}{\left( n+2 \right) \sum_{k\in E_n}{\left| g\left( k \right) \right|}v\left( k \right)}
\\
\leq& Cv\left( S \right) +C\sum_{n\in K^{\prime}}{\sum_{k\in E_n}{\left| g\left( k \right) \right|v\left( k \right)}}\log ^+\left| g\left( k \right) \right|
\\
\leq& C\left( v\left( S \right) +\sum_{k\in S}{\left| g\left( k \right) \right|v\left( k \right) \log ^+\left| g\left( k \right) \right|} \right).
\end{align*}
Taking $Tg(k) =M_{\alpha}(g\omega^{q\alpha})(k)$ and $g(k) =x(k) \omega (k)^{-q\alpha}$ into the above inequality, we finally obtain (\ref{eq4.2}).

Repeating the proof of \cite[Lemma 4.8]{hao2} with $k\in \mathbb{Z}$ replaced by $k\in S$, since $\omega^q\in \mathcal{A}_{\infty}$, we have
\begin{equation*}
\sum_{k\in S}{\left| I_{\alpha}x\left( k \right) \omega \left( k \right) \right|^q}\leq C\sum_{k\in S}{\left| M_{\alpha}x\left( k \right) \omega \left( k \right) \right|^q}.
\end{equation*}
From this and (\ref{eq4.2}), we immediately deduce (\ref{eq4.3}).
\end{proof}

The boundedness of $I_{\alpha}$ from $l_{\omega ^p}^{p}$ to $l_{\omega ^q}^{q}$ is proved by \cite[Theorem 4.4(i)]{hao2} for $\omega\in \mathcal{A}(p,q)$, but the related proof is complicated, and now, inspired by \cite[Theorem 1]{welland}, we can give another simple proof. Moreover, there is a gap in the proof of \cite[Theorem 4.4(ii)]{hao2}, precisely, $\|I_{\alpha}x\| _{l_{\omega ^q}^{q}}\sim\|M_{\alpha}x\| _{l_{\omega ^q}^{q}}$ is wrong for any $x\in l_{\omega ^p}^{p}(\mathbb{Z})$. We will correct the proof of \cite[Theorem 4.4(ii)]{hao2} in the following proof.

\begin{proof}[Proof of Theorem \ref{t1.4}]
Proof of (i). Let $R>0$ be a fixed real number. Then
\begin{equation*}
I_{\alpha}x\left( k \right) =\sum_{\left\{ i\in \mathbb{Z} :0<\left| k-i \right|\leq R \right\}}{\frac{x\left( i \right)}{\left| k-i \right|^{1-\alpha}}}+\sum_{\left\{ i\in \mathbb{Z} :\left| k-i \right|> R \right\}}{\frac{x\left( i \right)}{\left| k-i \right|^{1-\alpha}}}=:\mathrm{I+II}.
\end{equation*}
Let $\epsilon \in (0,\mathrm{min}\{\alpha,1-\alpha\})$. We have
\begin{equation*}
\left| \mathrm{I} \right|\leq \sum_{j=0}^{\infty}{\sum_{2^{-j-1}R< \left| k-i \right|\leq 2^{-j}R}{\frac{\left| x\left( i \right) \right|}{\left| k-i \right|^{1-\alpha}}}}\leq \sum_{j=0}^{\infty}{\left( 2^{-j-1}R \right) ^{\alpha -1}}\sum_{2^{-j-1}R< \left| k-i \right|\leq 2^{-j}R}{\left| x\left( i \right) \right|}.
\end{equation*}
We only consider the case $2^{-j}R>\frac{1}{2}$ for $j\in \mathbb{N}$, or else $\{2^{-j-1}R< |k-i|\leq 2^{-j}R \} =\emptyset$, $j\in \mathbb{N}$. Then we further have
\begin{align*}
\left| \mathrm{I} \right|\leq& \sum_{j=0}^{\infty}{\left( 2^{-j-1}R \right) ^{\alpha -1}\left[ 3\left( 2^{-j+1}R \right) \right] ^{1-\alpha +\epsilon}}\frac{1}{\left( 2\lfloor 2^{-j}R \rfloor +1 \right) ^{1-\alpha +\epsilon}}\sum_{\left| k-i \right|\leq \lfloor 2^{-j}R \rfloor}{\left| x\left( i \right) \right|}
\\
\leq& \frac{12^{1-\alpha +\epsilon}}{2^{\epsilon}-1}R^{\epsilon}M_{\alpha -\epsilon}x\left( k \right).
\end{align*}
Similarly, we only consider the case $2^{j}R>\frac{1}{2}$ for $j\in \mathbb{N}$, then we obtain
\begin{align*}
\left| \mathrm{II} \right|\leq& \sum_{j=1}^{\infty}{\sum_{2^{j-1}R< \left| k-i \right|\leq 2^jR}{\frac{\left| x\left( i \right) \right|}{\left| k-i \right|^{1-\alpha}}}}\leq \sum_{j=1}^{\infty}{\left( 2^{j-1}R \right) ^{\alpha -1}}\sum_{2^{j-1}R< \left| k-i \right|\leq 2^jR}{\left| x\left( i \right) \right|}
\\
\leq& \sum_{j=1}^{\infty}{\left( 2^{j-1}R \right) ^{\alpha -1}\left[ 3\left( 2^{j+1}R \right) \right] ^{1-\alpha -\epsilon}\frac{1}{\left( 2\lfloor 2^jR \rfloor +1 \right) ^{1-\alpha -\epsilon}}}\sum_{\left| k-i \right|\leq \lfloor 2^jR \rfloor}{\left| x\left( i \right) \right|}
\\
\leq& \frac{12^{1-\alpha -\epsilon}}{1-2^{-\epsilon}}R^{-\epsilon}M_{\alpha +\epsilon}x\left( k \right).
\end{align*}
Therefore,
\begin{equation*}
\left| I_{\alpha}x\left( k \right) \right|\leq \frac{12^{1-\alpha +\epsilon}}{2^{\epsilon}-1}R^{\epsilon}M_{\alpha -\epsilon}x\left( k \right) +\frac{12^{1-\alpha -\epsilon}}{1-2^{-\epsilon}}R^{-\epsilon}M_{\alpha +\epsilon}x\left( k \right).
\end{equation*}
Let $C_{\alpha}:=\max \left\{ \frac{12^{1-\alpha +\epsilon}}{2^{\epsilon}-1},\frac{12^{1-\alpha -\epsilon}}{1-2^{-\epsilon}} \right\}$. By choosing $R^{\epsilon}=(\frac{M_{\alpha +\epsilon}x(k)}{M_{\alpha -\epsilon}x(k)}) ^{\frac{1}{2}}$, we have
\begin{equation}\label{eq4.8}
\left| I_{\alpha}x\left( k \right) \right|\leq C_{\alpha}\left[ M_{\alpha +\epsilon}x\left( k \right) M_{\alpha -\epsilon}x\left( k \right) \right] ^{\small{\frac{1}{2}}}.
\end{equation}
Let $\frac{1}{q_{\epsilon}}=\frac{1}{p}-(\alpha +\epsilon)$ and $\frac{1}{\widetilde{q}_{\epsilon}}=\frac{1}{p}-(\alpha -\epsilon)$. If $\epsilon$ is small enough, by discrete reverse H\"older's inequality, then $\omega \in \mathcal{A}(p,q)$ implies $\omega \in \mathcal{A} (p,q_{\epsilon})$; by discrete H\"older's inequality, then $\omega \in \mathcal{A}(p,q)$ implies $\omega \in \mathcal{A}(p,\widetilde{q}_{\epsilon})$. Let $p_1:=\frac{2q_{\epsilon}}{q}>1$ and $p_2:=\frac{2\widetilde{q}_{\epsilon}}{q}>1$. Then
\begin{equation*}
\frac{1}{p_1}+\frac{1}{p_2}=\frac{q}{2q_{\epsilon}}+\frac{q}{2\widetilde{q}_{\epsilon}}=q\left( \frac{1}{p}-\alpha \right) =1,\enspace\frac{qp_1}{2}=q_{\epsilon}\enspace\mathrm{and}\enspace\frac{qp_2}{2}=\widetilde{q}_{\epsilon}.
\end{equation*}
From this, (\ref{eq4.8}), discrete H\"older's inequality and applying (\ref{eq1.1}) for the operators $M_{\alpha +\epsilon}$ and $M_{\alpha -\epsilon}$, we conclude that
\begin{align*}
&\sum_{k\in \mathbb{Z}}{\left| I_{\alpha}x\left( k \right) \omega \left( k \right) \right|^q}
\\
\leq& C_{\alpha}\sum_{k\in \mathbb{Z}}{\left[ M_{\alpha +\epsilon}x\left( k \right) \omega \left( k \right) \right] ^{\small{\frac{q}{2}}}}\left[ M_{\alpha -\epsilon}x\left( k \right) \omega \left( k \right) \right] ^{\small{\frac{q}{2}}}
\\
\leq& C_{\alpha}\left( \sum_{k\in \mathbb{Z}}{\left[ M_{\alpha +\epsilon}x\left( k \right) \omega \left( k \right) \right] ^{\small{\frac{qp_1}{2}}}} \right) ^{\frac{1}{p_1}}\left( \sum_{k\in \mathbb{Z}}{\left[ M_{\alpha -\epsilon}x\left( k \right) \omega \left( k \right) \right] ^{\small{\frac{qp_2}{2}}}} \right) ^{\frac{1}{p_2}}
\\
=&C_{\alpha}\left( \sum_{k\in \mathbb{Z}}{\left[ M_{\alpha +\epsilon}x\left( k \right) \omega \left( k \right) \right] ^{q_{\epsilon}}} \right) ^{\frac{1}{q_{\epsilon}}\frac{q_{\epsilon}}{p_1}}\left( \sum_{k\in \mathbb{Z}}{\left[ M_{\alpha -\epsilon}x\left( k \right) \omega \left( k \right) \right] ^{\widetilde{q}_{\epsilon}}} \right) ^{\frac{1}{\widetilde{q}_{\epsilon}}\frac{\widetilde{q}_{\epsilon}}{p_2}}
\\
\leq& C_{\alpha}\left( \sum_{k\in \mathbb{Z}}{\left| x\left( k \right) \omega \left( k \right) \right|^p} \right) ^{\frac{q_{\epsilon}}{pp_1}}\left( \sum_{k\in \mathbb{Z}}{\left| x\left( k \right) \omega \left( k \right) \right|^p} \right) ^{\frac{\widetilde{q}_{\epsilon}}{pp_2}}
\\
=&C_{\alpha}\left( \sum_{k\in \mathbb{Z}}{\left| x\left( k \right) \omega \left( k \right) \right|^p} \right) ^{\frac{q}{p}}.
\end{align*}

Proof of (ii). Let $S:=S_{m,N}$, $x(k):=\omega (k) ^{-p^{\prime}}\mathcal{X} _S(k)$ and $2\lambda =|S|^{\alpha-1}\sum_{k\in S}{| x(k)|}$. For every $k\in S$, we have $\overline{M}_{\alpha}x(k)\geq 2\lambda >\lambda$ and hence $S\subset \{k\in \mathbb{Z} :\overline{M}_{\alpha}x(k)>\lambda \}$. If $\omega$ is a discrete weighted sequence which is monotonic on $\mathrm{supp}\,\omega$, then $|x|$ is a monotonic sequence on $\mathrm{supp}\,x$. Then by Lemma \ref{l3.3}, Lemma \ref{l3.6} and that $I_\alpha$ is bounded from $l_{\omega ^p}^{p}$ to $l_{\omega ^q}^{q}$, we obtain
\begin{align*}
\lambda \left( \sum_{k\in S}{\omega \left( k \right) ^q} \right) ^{\frac{1}{q}}\leq& \lambda \left( \sum_{\left\{ k\in \mathbb{Z} :\overline{M}_{\alpha}x\left( k \right) >\lambda \right\}}{\omega \left( k \right) ^q} \right) ^{\frac{1}{q}}\leq \left( \sum_{k\in \mathbb{Z}}{\left| \overline{M}_{\alpha}x\left( k \right) \right|^q}\omega \left( k \right) ^q \right) ^{\frac{1}{q}}
\\
\leq& C\left( \sum_{k\in \mathbb{Z}}{\left| I_{\alpha}x\left( k \right) \right|}^q\omega \left( k \right) ^q \right) ^{\frac{1}{q}}\leq C\left( \sum_{k\in \mathbb{Z}}{\left| x\left( k \right) \right|^p\omega \left( k \right) ^p} \right) ^{\frac{1}{p}}.
\end{align*}
By $\frac{1}{q}=\frac{1}{p}-\alpha$ and taking $x(k):=\omega (k) ^{-p^{\prime}}\mathcal{X} _S(k)$ and $2\lambda =|S|^{\alpha-1}\sum_{k\in S}{| x(k)|}$ into the above inequality, we further obtain
\begin{equation*}
\left( \frac{1}{\left| S \right|}\sum_{k\in S}{\omega \left( k \right) ^q} \right) ^{\frac{1}{q}}\left( \frac{1}{\left| S \right|}\sum_{k\in S}{\omega \left( k \right) ^{-p^{\prime}}} \right) ^{\frac{1}{p^{\prime}}} \leq C,
\end{equation*}
which implies $\omega \in \mathcal{A}(p,q)$.
\end{proof}

\begin{remark}\label{r4.4}
The upper pointwise estimate (\ref{eq4.8}) of $I_{\alpha}$ related to $M_{\alpha}$ plays an important role in the proof of Theorem \ref{t1.4}. In fact, (\ref{eq4.8}) has a more general upper pointwise estimate. Let $0\leq \alpha _1<\alpha <\alpha _2<1$. Repeating the estimate of (\ref{eq4.8}) with $\alpha -\epsilon =\alpha _1$ and $\alpha +\epsilon =\alpha _2$, we obtain
\begin{equation*}
\left| I_{\alpha}x\left( k \right) \right|\leq C\left[ M_{\alpha _1}x\left( k \right) \right] ^{\frac{\alpha _2-\alpha}{\alpha _2-\alpha _1}}\left[ M_{\alpha _2}x\left( k \right) \right] ^{\frac{\alpha -\alpha _1}{\alpha _2-\alpha _1}},\enspace k\in \mathbb{Z}.
\end{equation*}
\end{remark}

The following Theorem \ref{t4.5} is a discrete variant of \cite[Theorem 2]{welland}, which is an interesting estimate.

\begin{theorem}\label{t4.5}
Let $0<\alpha <1$, $p=\frac{1}{\alpha}$, $\delta >0$, $q\in (\max \{p,p^{\prime}\},\infty )$ and $\omega \in \mathcal{A}(p,\infty)$. If $x$ is supported in some symmetric interval $S_{m,N}\subset \mathbb{Z}$, $N\in \mathbb{Z_+}$ and
\begin{equation}\label{eq4.9}
\left\| x \right\| _{l_{\omega ^p}^{p}}\leq \underset{k\in S_{m,N}}{\min}\omega \left( k \right).
\end{equation}
Then there exists a positive constant $C=C_{\delta ,q,\alpha}$ such that
\begin{equation*}
\sum_{k\in S_{m,N}}{\exp}\left[ \frac{1}{2}\left| \left| I_{\alpha}x\left( k \right) \right|-\delta \right|^{p^{\prime}} \right] \omega \left( k \right) ^q\le CN\left\| x \right\| _{l_{\omega ^p}^{p}}^{q}.
\end{equation*}
\end{theorem}

\begin{proof}
Let $S:=S_{m,N}$. For given $q\in (\max \{p,p^{\prime}\},\infty )$, define $\epsilon :=\frac{1}{q}$, where $\epsilon \in (0,\mathrm{min}\{\alpha,1-\alpha\})$ and $\frac{1}{q}=\frac{1}{p}-(\alpha -\epsilon)$. Then for $\omega\in \mathcal{A}(p,\infty)$, by the definition of $\omega\in \mathcal{A}(p,q)$, we have $\omega\in \mathcal{A}(p,q)$. Let $R\in (0,2N]$ to be fixed later. By checking the proof of Theorem \ref{t1.4}, we have $I_{\alpha}x(k) =:\mathrm{I+II}$ and
\begin{equation*}
|\mathrm{I}|\le C_{\alpha}R^{\epsilon}M_{\alpha -\epsilon}x\left( k \right).
\end{equation*}
For $k\in S$, by $\mathrm{supp}\,x\subset S$, discrete H\"older's inequality and $(\alpha -1) p^{\prime}=-1$, we obtain
\begin{align*}
\left| \mathrm{II} \right|\leq& \sum_{\left\{ i\in \mathbb{Z} :R< \left| k-i \right|\leq 2N \right\}}{\frac{\left| x\left( i \right) \right|}{\left| k-i \right|^{1-\alpha}}}
\\
\leq& \left( \underset{i\in S}{\mathrm{sup}}\,\,\frac{1}{\omega \left( i \right)} \right) \left( \sum_{\left\{ i\in \mathbb{Z} :R<\left| k-i \right|\le 2N \right\}}{\left| x\left( i \right) \omega \left( i \right) \right|^p} \right) ^{\frac{1}{p}}\left( \sum_{\left\{ i\in \mathbb{Z} :R<\left| k-i \right|\le 2N \right\}}{\left| k-i \right|^{\left( \alpha -1 \right) p^{\prime}}} \right) ^{\frac{1}{p^{\prime}}}
\\
\leq& \left( \underset{i\in S}{\mathrm{sup}}\,\,\frac{1}{\omega \left( i \right)} \right) \left\| x \right\| _{l_{\omega ^p}^{p}}\left( 2\sum_{\left\{ j\in \mathbb{Z} :R< j\leq 2N \right\}}{j^{\left( \alpha -1 \right) p^{\prime}}} \right) ^{\frac{1}{p^{\prime}}}
\\
\leq& \left( 2\sum_{\left\{ j\in \mathbb{Z} :R< j\leq 2N \right\}}{j^{-1}} \right) ^{\frac{1}{p^{\prime}}} \leq \left( 2\int_R^{2N}{y^{-1}dy} \right) ^{\frac{1}{p^{\prime}}}\leq \left[ 2\log ^+\left( \frac{2N}{R} \right) \right] ^{\frac{1}{p^{\prime}}}.
\end{align*}
For $\delta>0$, choose $R^{\epsilon}=\min \left\{ \delta [C_{\alpha}M_{\alpha -\epsilon}x(k)]^{-1},\left( 2N \right) ^{\epsilon} \right\}$. Noting that $\frac{1}{\epsilon}=q$, then
\begin{equation*}
\left| I_{\alpha}x\left( k \right) \right|\leq \delta +\left\{ 2\log ^+\left[ 2N\delta ^{-q}{C_{\alpha}}^q\left( M_{\alpha -\epsilon}x\left( k \right) \right) ^q \right] \right\} ^{\frac{1}{p^{\prime}}}.
\end{equation*}
Therefore,
\begin{equation}\label{eq4.10}
\exp \left[ \frac{1}{2}\left| \left| I_{\alpha}x\left( k \right) \right|-\delta \right|^{p^{\prime}} \right] \le 2N\delta ^{-q}{C_{\alpha}}^q\left( M_{\alpha -\epsilon}x\left( k \right) \right) ^q.
\end{equation}
Multiply both sides of (\ref{eq4.10}) by $\omega (k)^q$, sum both sides on $k\in S$ and apply (\ref{eq1.1}) for the operator $M_{\alpha -\epsilon}$ with $\omega\in \mathcal{A}(p,q)$, we have
\begin{align*}
&\sum_{k\in S}{\exp \left[ \frac{1}{2}\left| \left| I_{\alpha}x\left( k \right) \right|-\delta \right|^{p^{\prime}} \right]}\omega \left( k \right) ^q \leq 2N\delta ^{-q}{C_{\alpha}}^q\sum_{k\in S}{\left| M_{\alpha -\epsilon}x\left( k \right) \right|^q}\omega \left( k \right) ^q
\\
\leq& CN\left( \sum_{k\in S}{\left| x\left( k \right) \omega \left( k \right) \right|^p} \right) ^{\frac{q}{p}}\leq CN\left\| x \right\| _{l_{\omega ^p}^{p}}^{q}.
\end{align*}
We finish the proof of Theorem \ref{t4.5}.
\end{proof}

\begin{remark}\label{r4.6}
The condition (\ref{eq4.9}) in Theorem \ref{t4.5} is not hard to verify. In fact, for any sequence $x$ supported in some symmetric interval $S_{m,N}\subset \mathbb{Z}$, $N\in \mathbb{Z_+}$, then for $\widetilde{x}(k):=(\underset{k\in S_{m,N}}{\min}\omega(k)) x(k) \|x\| _{l_{\omega ^p}^{p}}^{-1}$, we have $\|\widetilde{x}\| _{l_{\omega ^p}^{p}}\leq\underset{k\in S_{m,N}}{\mathrm{min}}\ \omega(k)$.
\end{remark}

%
%
%

\end{document}